\documentclass[ebook,11pt,oneside,openany]{report}
\usepackage{graphicx, psfrag, amscd, amssymb,amsmath,amsthm}

\usepackage[nottoc,notlot,notlof]{tocbibind} 

\usepackage[utf8]{inputenc} 
\usepackage[T1]{fontenc} 

\usepackage{mathpazo} 

\usepackage{lineno,hyperref,tikz}
\usepackage{caption}
\usepackage{subcaption}
\usepackage{float}
\usepackage{pdfpages} 

\newtheorem{lem}{Lemma}[section]
\newtheorem{lem*}{Lemma}
\newtheorem{thm}[lem]{Theorem}
\newtheorem{pro}[lem]{Proposition}
\newtheorem{cor}[lem]{Corollary}
\newtheorem{exa}[lem]{Example}
\newtheorem{conj}{Conjecture}
\newtheorem{defi}[lem]{Definition}
\newtheorem{defi*}[lem*]{Definition}

\newtheorem{rem}[lem]{Remarks}

\newcommand{\exc}{{\mathsf{exc}}}

\newcommand{\Des}{{\mathsf{Des}}}

\newcommand{\fwex}{{\mathsf{fwex}}}
\newcommand{\cro}{{\mathsf{cro}}}
\newcommand{\nest}{{\mathsf{nest}}}
\newcommand{\croB}{{\mathsf{cro_B}}}
\newcommand{\nestB}{{\mathsf{nest_B}}}
\newcommand{\nega}{{\mathsf{neg}}}
\newcommand{\wex}{{\mathsf{wex}}}

\newcommand{\Alt}{{\mathsf{Alt}}}

\newcommand{\cs}{\mathsf{cs}}
\newcommand{\pat}{\mathsf{pat}}

\newcommand{\cab}{\mathsf{31\textnormal{-}2}}
\newcommand{\acb}{\mathsf{13\textnormal{-}2}}
\newcommand{\bca}{\mathsf{2\textnormal{-}31}}

\renewcommand{\SS}{{\mathcal{S}}}

\newcommand{\F}{\mathcal{F}}
\newcommand{\G}{\mathcal{G}}
\renewcommand{\H}{\mathcal{H}}
\newcommand{\M}{\mathcal{M}}

\newcommand{\T}{\mathcal{T}}
\newcommand{\U}{\mathsf{U}}
\renewcommand{\L}{\mathsf{L}}
\newcommand{\W}{\mathsf{W}}
\newcommand{\D}{\mathsf{D}}

\begin{document}




\begin{titlepage} 
	\newcommand{\HRule}{\rule{\linewidth}{0.5mm}} 
	
	\center 
	
	
	\textsc{\LARGE Department of Mathematics\\
	National Taiwan Normal University	
	}\\[1.5cm] 

	\vfill

	A thesis submitted in partial fulfilment of the requirements for the degree of\\
	\textsc{\Large Master of Science}\\[0.5cm] 
	in\\
	\textsc{\large Mathematics}\\[0.5cm] 
	
	
	\HRule\\[0.4cm]
	
	{\huge\bfseries Signed Countings of Type B and D Permutations
and $t,q$-Euler numbers}\\[0.4cm] 
	
	\HRule\\[1.5cm]
	
	
	\begin{minipage}{0.4\textwidth}
		\begin{flushleft}
			\large
			\textit{Author}\\
			Hsin-Chieh \textsc{Liao} 
		\end{flushleft}
	\end{minipage}
	~
	\begin{minipage}{0.4\textwidth}
		\begin{flushright}
			\large
			\textit{Supervisor}\\
			Dr. Sen-Peng \textsc{Eu} 
		\end{flushright}
	\end{minipage}
	
	
	
	\vfill\vfill\vfill 
	
	{\large\today} 
	
	
	 
	
	\vfill 
	
\end{titlepage}


\abstract{
   A classical result states that the parity balance of number of excedances of all permutations (derangements, respectively) of length $n$ is the Euler number. In 2010, Josuat-Verg\`{e}s gives a $q$-analogue with $q$ representing the number of crossings. We extend this result to the permutations (derangements, respectively) of type B and D. It turns out that the signed counting are related to the derivative polynomials of $\tan$ and $\sec$.
   
   Springer numbers defined by Springer can be regarded as an analogue of Euler numbers defined on every Coxeter group. In 1992  Arnol'd showed that the Springer numbers of classical types A, B, D count various combinatorial objects, called snakes. In 1999 Hoffman found that derivative polynomials of $\sec x$ and $\tan x$ and their subtraction evaluated at certain values count exactly the number of snakes of certain types. Then Josuat-Verg\`{e}s studied the $(t,q)$-analogs of derivative polynomials   $Q_n(t,q)$, $R_n(t,q)$ and showed that as setting $q=1$ the polynomials are enumerators of snakes with respect to the number of sign changing. Our second result is to find a combinatorial interpretations of   $Q_n(t,q)$ and $R_n(t,q)$ as  enumerator of the snakes, although the outcome is somewhat messy.\\

}
\tableofcontents

\chapter{Motivation of the problems}
In this chapter the classical signed counting results of Euler and Roselle on permutations and derangements with respect to Eulerian statistics are presented. Then we introduce  Josuat-Verg\`{e}s'  $q$-analogues with $q$ representing the number of crossings. These classical signed counting results serve as motivations of our study.\\
\section{Signed countings on Permutations and Derangements}

Let $\mathfrak{S}_n$ denote the set of permutations on $[n]:=\{1,2,\ldots,n\}$ and $\mathfrak{S}_n^*$ denote the set of derangements on $[n]$. A permutation $\sigma\in\mathfrak{S}_n$ is a bijection on $[n]$ and we may write $\sigma$ as $\sigma_1\sigma_2\ldots\sigma_n$ ($\sigma_i\in [n]$) if $\sigma(i)=\sigma_i$  for all $1\le i\le n$. This is called the one-line natation of the permutation $\sigma$.

\begin{defi}
For a permutation
$\sigma=\sigma_1\sigma_2\cdots\sigma_n\in\mathfrak{S}_n$, an \emph{excedance} (\emph{weak excedance}, respectively) is an integer $i\in [n]$ such that $\sigma_i>i$ ($\sigma_i\ge i$, respectively). Let $\exc(\sigma)$ and $\wex(\sigma)$ denote the number of excedances and the number of weak excedances of $\sigma$, respectively.
\end{defi}

An elementary result is that the statistics $\exc$ and $\wex$ have same distribution in $\mathfrak{S}_n$ and $\sum_{\sigma\in\mathfrak{S}_n}y^{\wex(\sigma)}=y\sum_{\sigma\in\mathfrak{S}_n}y^{\exc(\sigma)}$. The polynomial $A_n(y)=\sum_{\sigma\in\mathfrak{S}_n}y^{\wex(\sigma)}$ is called the \emph{Eulerian polynomial} and $A_{n,k}=\#\{\sigma\in\mathfrak{S}_n| \wex(\sigma)=k\}$ is called the \emph{Eulerian number}.

\begin{defi}
The classical \emph{Euler numbers} $E_n$ are defined by
\[
\sum_{n\ge 0}E_n\frac{x^n}{n!}=\tan x+\sec x=1+x+\frac{x^2}{2!}+2\frac{x^3}{3!}+5\frac{x^4}{4!}+16\frac{x^5}{5!}+\ldots.
\]
The numbers $E_{2n}$ are called the \emph{secant numbers} and the numbers $E_{2n+1}$ are called the \emph{tangent numbers}.
\end{defi}
It is well-known that $E_n$ counts the number of \emph{alternating permutations} in $\mathfrak{S}_n$, i.e., $\sigma\in\mathfrak{S}_n$ such that
$\sigma_1>\sigma_2<\sigma_3>\ldots \sigma_n$. For example, when $n=3$, there are $2$ alternating permutations $213,312$; when $n=4$ there are $5$ alternating permutations $2143, 3142, 3241, 4132, 4231$. A permutation $\sigma$ satisfying $\sigma_1<\sigma_2>\sigma_3<\ldots \sigma_n$ is called a reverse alternating permutation.

An interesting result states that when we evaluate the Eulerian polynomial $A_n(y)$ at $y=-1$ depending on the parity of $n$  we either obtain $0$ or tangent numbers.
\begin{thm}[Euler\cite{Euler}; Foata, Sch\"{u}tzenberger\cite{Foata-Schu}]
\begin{equation}\label{Eulercan1}
    \sum_{\sigma\in\mathfrak{S}_n}(-1)^{\exc(\sigma)}=
    \left\{\begin{array}{ll}
        0,  & \mbox{if $n$ is even,}\\
    (-1)^{\frac{n-1}{2}}E_n, & \mbox{if $n$ is odd.}
    \end{array}
    \right.
\end{equation}
\end{thm}
This identity was first discovered by Euler\cite{Euler} in a different form when he introduced Eulerian polynomials, the presenting form of the identity was obtained by Foata and Sch\"{u}tzenberger\cite{Foata-Schu}. Interestingly, the other half of the result shows up while we restrict our attention on the derangements in $\mathfrak{S}_n$.
\begin{thm}[Roselle\cite{Roselle1968}]
\begin{equation}\label{Eulercan2}
    \sum_{\sigma\in\mathfrak{S}_n^{*}}(-1)^{\exc(\sigma)}=
    \left\{
            \begin{array}{ll}
        (-1)^{\frac{n}{2}}E_n,  & \mbox{if $n$ is even,}\\
        0, & \mbox{if $n$ is odd.}
            \end{array}
    \right.
\end{equation}
\end{thm}
\noindent This was first obtained by Roselle\cite{Roselle1968} using a slightly different combinatorial interpretation.

\section{$q$-analogue of the signed counting identities}
As we can see in $(\ref{Eulercan1}),(\ref{Eulercan2})$, both sides of the identities occur in $\mathfrak{S}_n$, it is natural to seek $q$-analogues of $(\ref{Eulercan1}),(\ref{Eulercan2})$.
In fact there are three different q-analogues that have been discovered by Foata and Han \cite{FH}, Josuat-Verg\`{e}s\cite{JVerges2010}, Shin and Zeng \cite{Shin2010} respectively. In this section we will introduce the one obtained by Josuat-Verg\`{e}s.\\

To begin with, we need to introduce the corresponding $q$-analogous of Eulerian polynomials and Euler numbers.
\begin{defi}
A \emph{crossing} of a permutation $\sigma=\sigma_1\sigma_2\ldots\sigma_n$ is a pair of $(i,j)$ $(1\le i<j\le n)$ such that $i<j\le\sigma_i<\sigma_j$ or $\sigma_i<\sigma_j<i<j$. We denote by $\cro(\sigma)$ the number of crossings in $\sigma$.
\end{defi}
Crossings of a permutation can be visualized via permutation diagram, see Figure \ref{diagram1}. Let $\sigma=6453172$ then the crossing of $\sigma$ are $(2,3),(1,6), (5,7)$, hence $\cro(\sigma)=3$.\\
\begin{figure}[t!]
\begin{center}
	\begin{tikzpicture}[scale=.4]
		\node [fill,circle,inner sep=1.5pt] (b1) at (0,0){};
		\node [fill,circle,inner sep=1.5pt] (b2) at (2,0){};
		\node [fill,circle,inner sep=1.5pt] (b3) at (4,0){};
		\node [fill,circle,inner sep=1.5pt] (b4) at (6,0){};
		\node [fill,circle,inner sep=1.5pt] (b5) at (8,0){};
		\node [fill,circle,inner sep=1.5pt] (b6) at (10,0){};
		\node [fill,circle,inner sep=1.5pt] (b7) at (12,0){};
		
	\path 
		(b1) edge [->,out=45, in=135] node [pos=.5,above]  {} (b6)
		(b2) edge [->,out=45, in=135] node [pos=.5,above]{} (b4)
		(b3) edge [->,out=45,in=135] node [pos=.5,above]{} (b5)
		(b4) edge [->,below, out=225, in=315] node [pos=.5,below] {} (b3)
		(b5) edge [->,below, out=225,in=315] node [pos=.5,below] {} (b1)
		(b6) edge [->,out=70,in=110] node [pos=.5,above] {} (b7)
		(b7) edge [->,below, out=225,in=315] node [pos=.5,below] {} (b2);
		
		\node at (0,-0.75)  {\tiny $1$};
		\node at (2,-0.75)  {\tiny $2$};
		\node at (4,-0.75)  {\tiny $3$};
		\node at (6,-0.75)  {\tiny $4$};
		\node at (8,-0.75)  {\tiny $5$};
		\node at (10,-0.75)  {\tiny $6$};
		\node at (12,-0.75)  {\tiny $7$};
	\end{tikzpicture}
\end{center}
\caption{The diagram of $\sigma=6453172$}
\label{diagram1}
\end{figure}
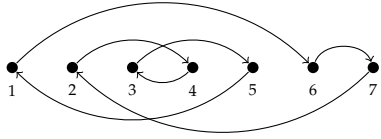

Then we have a $q$-analogue of Eulerian numbers
\begin{equation}\label{q-A}
	A_{n,k}(q)=\sum_{\substack{\sigma\in\mathfrak{S}_n\\ 			\wex(\sigma)=k}}q^{\cro(\sigma)}
\end{equation}
and the corresponding $q$-Eulerian polynomials
\[
	A(y,q)=\sum_{k=1}^nA_{n,k}(q)y^k=\sum_{\sigma\in\mathfrak{S}_n}y^{\wex(\sigma)}q^{\cro(\sigma)}.
\]
The notion of crossings of a permutation was first introduced by Williams \cite{WilliamsGrs} along with another notion called \emph{alignments} in the study of totally positivity Grassmann cells. In \cite{WilliamsGrs}, Williams also define $A_{n,k}(q)$  though in terms of alignments. But a simple relation between the number of aligenments and the number of crossings shown later by Corteel \cite{Corteel-cross} gives the equivalent definition in (\ref{q-A}).

The following $q$-analogue of Euler number was introduced by Han, Randrianarivony, Zeng \cite{Han-Zeng}.
\begin{defi}
The \emph{$q$-tangent numbers} $E_{2n+1}(q)$ are defined by 
\begin{equation} \label{eqn:q-tangent}
	\sum_{n=0}^{\infty}E_{2n+1}(q)z^n=
	\frac{1}{1-\cfrac{[1]_q[2]_qz}{1-\cfrac{[2]_q[3]_qz}{1-\cfrac{[3]_q[4]_qz}{\ddots}}}}
\end{equation}
and the \emph{$q$-secant numbers} $E_{2n}(q)$ are defined by
\begin{equation}\label{eqn:q-secant}
	\sum_{n=0}^{\infty}E_{2n}(q)z^n=
	\frac{1}{1-\cfrac{[1]_q^2z}{1-\cfrac{[2]_q^2z}{1-\cfrac{[3]_q^2z}{\ddots}}}}
\end{equation}
\end{defi}
The first few polynomials are $E_0(q)=E_1(q)=E_2(q)=1,$ $E_3(q)=1+q$,  $E_4(q)=2+2q+q^2$, $E_5(q)=2+5q+5q^2+3q^3+q^4$.

\noindent The polynomial $E_n(q)$ has a combinatorial interpretation \cite{Chebikin,JVerges2010}:
\[
	E_n(q)=\sum_{\sigma\in\Alt_n}q^{\cab(\sigma)}
\]
where $\Alt_n$ is the set of alternating permutations of length $n$ and $\cab(\sigma)=\#\{(i,j):i+1<j,\sigma_{i+1}<\sigma_j<\sigma_i\}$.

Using the above $E_n(q)$ and the number of crossings $\cro$, Josuate-Verg\`{e}s \cite{JVerges2010} derived $q$-analogs of Eqs (\ref{Eulercan1}) and (\ref{Eulercan2}).
\begin{thm}[Josuate-Verg\`{e}s \cite{JVerges2010}]
\begin{equation}\label{JV1} 
	 \sum_{\pi\in\mathfrak{S}_n}(-1)^{\wex(\pi)}q^{\cro(\pi)}
	=\left\{\begin{array}{ll}
			0  						& \mbox{if }n\mbox{ is even,}\\			     (-1)^{\frac{n+1}{2}}E_n(q) & \mbox{if }n\mbox{ is odd;} 
   	\end{array}\right.	\\
\end{equation}
and
\begin{equation} \label{JV2}
	 \sum_{\pi\in\mathfrak{S}_n^*}(-\frac{1}{q})^{\wex(\pi)}q^{\cro(\pi)}=\left\{\begin{array}{ll}
			(-\frac{1}{q})^{\frac{n}{2}}E_n(q)  & \mbox{if }n\mbox{ is even,}\\
			0 & \mbox{if }n\mbox{ is odd.} 		
	\end{array}\right. 
\end{equation}
\end{thm}
\begin{exa}
When $n=4$, from table \ref{n=4} we have
\begin{align*}
		\sum_{\pi\in\mathfrak{S}_n^*}(-\frac{1}{q})^{\wex(\pi)}q^{\cro(\pi)}=&\frac{1}{q^2}-\frac{1}{q}+\frac{1}{q}+\frac{1}{q}+1+\frac{1}{q}+\frac{1}{q}-\frac{1}{q}+\frac{1}{q}+\frac{1}{q^2} \\
		=&\frac{2}{q^2}+\frac{2}{q}+1\\
		=&(-\frac{1}{q})^2(2+2q+q^2)\\
		=&(-\frac{1}{q})^2E_4(q)
		\end{align*}
\end{exa}

\begin{table}[htb]
		\centering
		\begin{tabular}[t]{c|c c}
		  $\mathfrak{S}_4^*$ &  $\wex$ & $\cro$\\
		  \hline
		  2143	&    2  &   0 \\
		  2341	&    3  &   2 \\
		  2413	&    2  &   1 \\
		  3142	&    2  &   1 \\
		  3412	&    2  &   2 \\
		  4123	&    1  &   0 \\
		  4132	&    2  &   1 \\
		  4312	&    2  &   1 \\
		  4321	&    2  &   0 
		\end{tabular}
		\begin{tabular}[t]{c|c}
		  $\Alt_4$ &  $\cab$\\
		  \hline
		  2143	&  0 \\
		  3142	&  1 \\
		  3241	&  0 \\
		  4132  &  2 \\
		  4231  &  1
		\end{tabular}
		\caption{}
		\label{n=4}
\end{table}

Note that the symmetric group $\mathfrak{S}_n$ is just the finite irreducible Coxeter group of type A. In type B and type D, there are combinatorial models similar to permutations. Fortunately, the notions we have mentioned, for instance $\wex$, $\cro$, also have type B analogues. One of our purpose in this work is to extend the results of (\ref{Eulercan1}),(\ref{Eulercan2}) to type B and D.

\chapter{Signed Permutations and Snakes}
  In this chapter we introduce the type B and D analogues of several notions including signed and even signed permutations, flag weak excedence, crossings of type B, and the corresponding enumerators. Then we briefly describe the type-free analogue of Euler number-Springer numbers. We focus on the Springer number of type B. The combinatorial model of $S_n$, which is called the snakes of type B, are introduced along with other types of snakes. At last we present the connection between snakes and the derivative polynomials of $\tan$ and $\sec$.\\

\section{Signed Permutations}

The type B and type D analogs of permutations are signed and even sigend permutations respectively. 
\begin{defi}\noindent
\begin{itemize}
\item[(i)] A \emph{signed permutation} of $[n]$ is a bijection $\sigma$ of the set $[\pm n]:=\{-n,-n+1,\dots,-1,1,2,\dots,n\}$ onto itself such that $\sigma(-i)=-\sigma(i)$ for all $i\in[\pm n]$. For convenience, we write $-i$ as $\bar{i}$. Sometimes we denote $\sigma$ as $\sigma_1,\sigma_2,\ldots,\sigma_n$, which is called the \emph{window notation} of $\sigma$, where $\sigma_i=\sigma(i)$ for $1\le i \le n$.
\item[(ii)] An \emph{even signed permutation} is a signed permutation with even number of negative entries in its window notation.
\end{itemize}
\end{defi}
Denote $B_n$ and $D_n$ the set of signed permutations and even signed permutations of $[n]$, and $B_n^*$ ($D_n^*$ respectively) the subset of $B_n$ ($D_n$ respectively) without fixed points.\\
For example, $B_2=\{12,\bar{1}2,1\bar{2},\bar{1}\bar{2},21,2\bar{1},\bar{2}1,\bar{2}\bar{1}\}$, $D_2=\{12,\bar{1}\bar{2},21,\bar{2}\bar{1}\}$, $B_2^*=\{\bar{1}\bar{2}, 21,2\bar{1},\bar{2}1,\bar{2}\bar{1}\}$, $D_2^*=\{\bar{1}\bar{2}, 21,\bar{2}\bar{1}\}$. 
 
The type B analogous of weak excedance we need is the \emph{flag weak excedance} of signed permutations, which is defined as following. 
\begin{defi}[Flag weak excedance]\noindent
For $\sigma\in B_n$, we define $\wex(\sigma)=\#\{i\in [n]: \sigma_i\ge i\}$ and $\nega(\sigma)=\#\{\sigma_i: i\in[n], \sigma_i<0\}$.  Then the \emph{flag weak excedance number} is defined as
\[
	\fwex(\sigma)=2\wex(\sigma)+\nega(\sigma).
\]
\end{defi}

\section{Crossing of type B}
Notion of crossings of signed permutations were given by Corteel, Josuat-Verg\`{e}s and Williams in \cite{CJW} and was studied further in the later work of Corteel, Josuat-Verg\`{e}s and Kim \cite{CJK}. The defintion is defined as following.
\begin{defi}[Crossings of type B]
For $\sigma=\sigma_1\sigma_2\cdots\sigma_n\in B_n$, a \emph{crossing} of $\sigma$ is a pair $(i,j)$  with $i,j\ge 1$ such that
\begin{itemize}
  \item $i<j\le \sigma_i< \sigma_j$ or
  \item $-i<j\le -\sigma_i < \sigma_j$ or
  \item $i>j>\sigma_i>\sigma_j$.
\end{itemize}
\end{defi}

Similar to the case in type A, we may represent a signed permutation in $B_n$ by diagram which makes the number of crossings easilier to count. Corteel et al. \cite{CJK} offer several ways to do it. One of them is through the \emph{full pignose diagram}, see Figure \ref{FullPigNose} for an example of $\sigma=6\bar{3}\bar{5}14\bar{7}\bar{2}$.\\

\noindent\textbf{Construct the diagram:} We assign two vertices for each $i\in[\pm n]$ and arrange the vertices in a line as in the figure.
\begin{itemize}
		\item For each $i>0$ if $\sigma_i>0$ ($\sigma_i<0$, resp.), then we connect the first vertex of $i$ to the second vertex of $\sigma_i$ (first vertex of $\sigma_i$, resp.) with an arc in the following way: draw the arc above the horizontal line if $i\le\sigma_i$, and below the horizontal line if $i>\sigma_i$.
		\item for each $i<0$ if $\sigma_i<0$ ($\sigma_i>0$, resp.), then we connect the second vertex of $i$ to the first vertex of $\sigma_i$ (second vertex of $\sigma_i$, resp.) with an arc in the following way: draw the arc below the horizontal line if $i\ge\sigma_i$, and above the horizontal line if $i<\sigma_i$.
\end{itemize} 
  We can see that the configuration of upper arcs and that of lower arcs are symmetric, and the number of crossings of $\sigma$ is exactly the number of crossings between the arcs above the horizontal line. 
\begin{exa}  
Let $\sigma=6\bar{3}\bar{5}14\bar{7}\bar{2}$, then the crossings are $(7,1),(3,1),(2,1)$ ($-i<j\le -\sigma_i < \sigma_j$) and $(4,2)$,$(4,3)$,$(7,2)$,$(7,3)$,$(7,6)$ ($i>j>\sigma_i>\sigma_j$) so $\croB(\sigma)=8$, see figure \ref{FullPigNose} again.

\begin{figure}[ht]
\centering
\begin{tikzpicture}[scale=.6]
		\footnotesize
1		\foreach \x in {1,2,3,4,5,6,7}{
			\node[shape=circle,inner sep=1pt,draw,fill] (A\x) at (\x-.15,0) {};
			\node[shape=circle,inner sep=1pt,draw,fill] (B\x) at (\x+.15,0) {};
			\node[shape=circle,inner sep=1pt,draw,fill] (-A\x) at (-\x+.15,0) {};
			\node[shape=circle,inner sep=1pt,draw,fill] (-B\x) at (-\x-.15,0) {};
			\draw[gray] (\x,0) ellipse (.3 and .2);
			\draw[gray] (-\x,0) ellipse (.3 and .2);
			\draw (\x,-.4) node{$\x$};
			\draw (-\x,-.4) node{$-\x$};}
	
		\draw[thick] (A1) to [out=90,in=90] (B6);
		\draw[thick] (-A1) to [out=-90,in=-90] (-B6);
		\draw[thick] (A2) to [out=-90,in=-90] (-B3);
		\draw[thick] (-A2) to [out=90,in=90] (B3);
		\draw[thick] (A3) to [out=-90,in=-90] (-B5);
		\draw[thick] (-A3) to [out=90,in=90] (B5);
		\draw[thick] (A4) to [out=-90,in=-90] (B1);
		\draw[thick] (-A4) to [out=90,in=90] (-B1);
		\draw[thick] (A5) to [out=-90,in=-90] (B4);
		\draw[thick] (-A5) to [out=90,in=90] (-B4);
		\draw[thick] (A6) to [out=-90,in=-90] (-B7);
		\draw[thick] (-A6) to [out=90,in=90] (B7);
		\draw[thick] (A7) to [out=-90,in=-90] (-B2);
		\draw[thick] (-A7) to [out=90,in=90] (B2);
	\end{tikzpicture}
	\caption{Full pignose diagram for $\sigma=6\bar{3}\bar{5}14\bar{7}\bar{2}$}
	\label{FullPigNose}	
\end{figure}
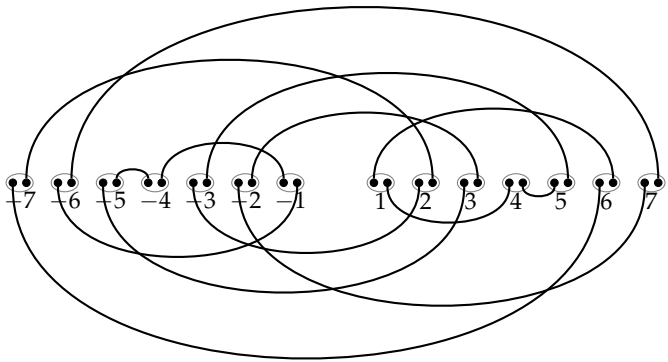
\end{exa}

\section{Refined Enumeration on Singed Permutations}

Let $B_n(y,t,q)=\sum_{\sigma\in B_n}y^{\fwex(\sigma)}t^{\nega(\sigma)}q^{\croB(\sigma)}$. The first few values are 
\begin{align*}
B_0(y,t,q)=& 1,\\
B_1(y,t,q)=& y^2+yt,\\ 
B_2(y,t,q)=& y^4+(2t+tq)y^3+(t^2q+t^2+1)y^2+ty.
\end{align*}
In particular, when we take $t=0$, we have $B_n(y,0,q)=A_n(y,q)$.

Corteel, Josuat-Verg\`{e}s and  Williams \cite{CJW} showed that $B_n(y,t,q)$ is also a generating function of permutation tableaux of type B. Using permutation tableaux, Corteel et al.\cite{CJW} proved that $B_n(y,t,q)$ satisfied a \emph{Matrix Ansatz}. (They only consider $B_n(1,t,q)$, but their proofs work for $B_n(y,t,q)$.)
\begin{thm}[\cite{CJW}]
Let $D$ and $E$ be matrices, $\langle W|$ a row vector, $|V\rangle$ a column vector, satisfying
\[
	D|V\rangle=|V\rangle, \qquad \langle W|E=yt\langle W|D, \qquad DE=qED+D+E.
\]
Then 
\[
	B_n(y,t,q)=\langle W|(y^2D+E)|V\rangle.
\]
\end{thm}
By finding a solution to this \emph{Matrix Ansatz}, Corteel, Josuat-Verg\`{e}s and Kim  \cite{CJW} obtained the generating fucntion of $B_n(y,t,q)$ in the form of \emph{J-fractions} (Jacobi continued fractions).\\

\begin{defi} \label{def:J-fraction} For any two sequences $\{\mu_h\}_{h\ge 0}$ and $\{\lambda_h\}_{h\ge 1}$, let $\mathfrak{F}(\mu_h,\lambda_h)$ denote the continued fraction
\begin{equation*}
\mathfrak{F}(\mu_h,\lambda_h)=\frac{1}{1-\mu_0 x}{{}\atop{-}}\frac{\lambda_1 x^2}{1-\mu_1 x}{{}\atop{-}}\frac{\lambda_2 x^2}{1-\mu_2 x}{{}\atop{-}}\frac{\lambda_3 x^2}{1-\mu_3x}{{}\atop{-}}\cdots
\end{equation*}
\end{defi}

\begin{thm}\cite{CJK}\label{thm:B-enumerator-cf}
The continued fraction expansion for the generating fucntion of $B_n(y,t,q)$ is
\begin{equation}  \label{eqn:B-enumerator-cf}
  \sum_{n\ge 0} B_n(y,t,q) x^n = \mathfrak{F}(\mu_h,\lambda_h)\end{equation}
where $\mu_h=y^2[h+1]_q+[h]_q+ytq^{h}([h]_q+[h+1]_q)$ for $h\ge 0$ and $\lambda_h=[h]_q^2(y^2+ytq^{h-1})(1+ytq^h)$ for $h\ge 1$
\end{thm}

\section{Gerneralized Euler numbers: Springer numbers}

In 1971, Springer \cite{Springer} defined an integer $K(W)$ ,which is usually called Springer number now, for any Coxeter group $W$. He also computed the quantity for all finite irreducible Coxeter systems. In particular, $\mathfrak{S}_n$ is the irreducible Coxeter group of type $A_{n-1}$ and $K(A_{n-1})=K(\mathfrak{S}_n)=E_n$.

\begin{defi}
Let $(W,S)$ be a Coxeter system, for any $w\in W$ the (right) descent set of $w$ is defined to be
\[
	\Des(w)=\{s\in S: \ell(ws)<\ell(w)\}.
\]
Let $J\subset S$ and $D_J=\{w\in W: \Des(w)=J\}$, then the \emph{Springer number} of $W$ is defined to be the cardinality of largest descent class
\[
	K(W):=\max_{J\subset S}|D_J|
\]
\end{defi}

\begin{exa} Let $W=\mathfrak{S}_3$ and $S=\{s_1=(12),s_2=(23)\}$, then when $J=\{s_1\}$ or $\{s_2\}$, $D_J=\{213,312\}$ or $\{132,231\}$ attains its maximun size. In this case, $K(\mathfrak{S}_3)=2=E_3$.
		\[\begin{array}{l|c}
				w\in A_2		&  \Des(w)\\
				\hline
				123=id   		&  \emptyset\\
				132=s_2			&  s_2  \\
				213=s_1 	    &  s_1  \\
				231=s_1s_2      &  s_2	\\
				312=s_2s_1		&  s_1	\\
				321=s_1s_2s_1	& s_1,s_2  
				\end{array}\]	
For general $\mathfrak{S}_n$, we can see that the descent class $D_J$ attain its maximun size either when $J=\{s_1,s_3,\ldots\}$ or $J=\{s_2,s_4\ldots\}$. And $D_J$ is exactly $\mathsf{RAlt}_n$ and $\mathsf{Alt}_n$, where $\mathsf{RAlt_n}$ is the set of reverse alternating permutations in $\mathfrak{S}_n$. For more values of $K(W)$ of classical types, see table \ref{TableSpringer}, in which we denote $S_n=K(B_n)$ and $S_n^D=K(D_n)$.
\end{exa}

\begin{table}
\begin{center}
 \begin{tabular}{|c| c c c c c c c c|} 
 \hline
  $n$  & 0 & 1 & 2 & 3 & 4 & 5 & 6 & $\ldots$ \\ [0.5ex] 
 \hline
 $E_n$ & 1 & 1 & 1 & 2 & 5 & 16 & 61 & $\ldots$ \\ 
 \hline
 $S_n$ & 1 & 1 & 3 & 11& 57& 361& 2763 & $\ldots$\\
 \hline
 $S_n^D$ & 1 & 1 & 1 & 5 & 23 & 151 & 1141 & $\ldots$\\
 \hline
\end{tabular}
\end{center}
\caption{Springer number of type $A$, $B$ and $D$}
\label{TableSpringer}
\end{table}

\section{Snakes of type B}
By describing Springer number geometrically in terms of Weyl chambers, Arnol'd \cite{Arnold} showed that for irreducible Coxeter systems of type $A$, $B$, $D$ the Springer number counts various types of \emph{snakes} (up-down permutations and up-down signed permutations). Our study mainly focuses on the snakes of type $B$. 
\begin{defi}
Let $\sigma=\sigma_1\ldots\sigma_n\in B_n$.
\begin{itemize}
\item The signed permutation $\sigma$ is a snake if $\sigma_1>\sigma_2<\sigma_3>\ldots\sigma_n$. Let $\mathcal{S}_n\subset B_n$ be the set of snakes of size $n$.
\item Let $\SS_n^{0}\subset \SS_n$ be the subset consisting of the snakes $\sigma$ with $\sigma_1>0$.
\item Let $\SS_n^{00}\subset\SS_n^{0}$ be the subset consisting of the snakes $\sigma$ with $\sigma_1>0$ and $(-1)^n\sigma_n<0$.
\end{itemize}
\end{defi}
\begin{exa}
For example, as $n=2$ then 
\[\SS_2=\{1\bar{2},\bar{1}\bar{2},21,2\bar{1}\},\qquad\SS^0_2=\{1\bar{2},21,2\bar{1}\}, \qquad \SS_2^{00}=\{1\bar{2},2\bar{1}\}.
\]
\end{exa}
Note that $\SS_n^0$ is the subset of snakes of type $B$ so $|\SS_n^0|=S_n$. The snakes in $\SS_n$ are introduced by Arnol'd \cite{Arnold}  under the name \emph{$\beta$-snakes} to study snakes of type $B$ and $D$, with $|\SS_n|=2^nE_n$. The subset $\SS_n^{00}$ is a variant introduced by Josuat-Verg\`{e}s \cite{JV} with $|\SS_n^{00}|=2^{n-1}E_n$.

There is another surprising link between snakes and derivative polynomials of trigonometric fuctions. Hoffman \cite{Hoffman} and Josuat Verg\`{e}s \cite{JV} studied the polynomials $P_n(t)$, $Q_n(t)$ and $R_n(t)$ , which are defined as following
\[
	\frac{d^n}{dx^n}\tan x=P_n(\tan x),~~ \frac{d^n}{dx^n}\sec x=Q_n(\tan x)\sec x, ~~\frac{d^n}{dx^n}\sec^2 x=R_n(\tan x)\sec^2 x.
\]
Hoffman \cite{Hoffman} showed that $P_n(1)=2^nE_n$, $Q_n(1)=S_n$ and $P_n(1)-Q_n(1)=S_n^D$. Then Josuat-Verg\`{e} \cite{JV} defined the polynomail $R_n(t)$ and proved that $R_n(1)=2^nE_{n+1}$; morveover, he gave combinatorial interpretations to $P_n(t)$, $Q_n(t)$ and $R_n(t)$ in terms of the distributions of number of changes of sign $\cs$ on $\SS_n, \SS_n^0$ and $\SS_n^{00}$.

\begin{defi}
For a snake $\sigma=\sigma_1\sigma_2\cdots\sigma_n\in B_n$, let $\cs(\sigma)$ denote the number of changes of sign through the entries $\sigma_1,\sigma_2,\dots,\sigma_n$, i.e., $\cs(\sigma):=\#\{i: \sigma_i\sigma_{i+1}<0,  0\le i\le n\}$ with the following convention for the entries $\sigma_0$ and $\sigma_{n+1}$:
\begin{itemize}
  \item $\sigma_0=-(n+1)$ and $\sigma_{n+1}=(-1)^n (n+1)$ if $\sigma\in\SS_n$;
  \item $\sigma_0=0$ and $\sigma_{n+1}=(-1)^n (n+1)$ if $\sigma\in\SS_n^{0}$;
  \item $\sigma_0=0$ and $\sigma_{n+1}=0$ if $\sigma\in\SS_n^{00}$.
\end{itemize}
\end{defi}
\begin{thm}[Josuat-Verg\`{e}s \cite{JV}]
For all $n\ge 0$, we have
\[
	P_n(t)=\sum_{\sigma\in\SS_n} t^{\cs(\sigma)}, \qquad Q_n(t)=\sum_{\sigma\in\SS_n^{0}} t^{\cs(\sigma)}, \qquad R_n(t)=\sum_{\sigma\in\SS_{n+1}^{00}} t^{\cs(\sigma)}.
\]
\end{thm}
\noindent
Let $D$ be the ordinary differentiate operator, and $U$ be the operator of multiplying $t$ on the vector space of polynomials $\mathbb{C}[t]$. From the recurrence when differentiating $\tan x$ and $\sec x$, we can easily see that the definition of polynomials $P_n, Q_n$ and $R_n$ can be rephrased as
\begin{align*}
	& P_n(t)=(D+UUD)^nt \\
	& Q_n(t)=(D+UDU)^n1 \\
	& R_n(t)=(D+DUU)^n1
\end{align*}
for $n\ge 0$.
 Josuat-Verg\`{e}s \cite{JV} defined the $q$-analogs of the derivative polynomials $Q_n$ and $R_n$ via the \emph{$q$-derivative}. Let $D$ be the $q$-analog of the differential operator acting on polynomials $f(t)$ by
\begin{equation}\label{q-derivative}
(D f)(t) := \frac{f(qt)-f(t)}{(q-1)t}
\end{equation}
and $U$ be the operator acting on $f(t)$ by multiplication by $t$. Notice that the $q$-derivative $D(t^n)=[n]_q t^{n-1}$ and the communication relation $DU-qUD=1$ hold.\\
\noindent Note that for $n\ge 1$
\begin{align*}
	P_n &=((I+UU)D)^{n-1}(I+UU)1=(I+UU)(D+DUU)^{n-1}1\\
		&=(I+UU)R_{n-1}=(1+t^2)R_{n-1}.
\end{align*}
Hence we have $P_{n}(t)=(1+t^2)R_{n-1}(t)$ for $n\ge 1$. This relation also holds for the $q$-analogs of $P_n$ and $R_n$, therefore we only focus on q-analogs of $Q_n$ and $R_n$.\\
\begin{defi}
The $q$-analogs of $Q_n$ and $R_n$ are defined algebraically by
\begin{equation}
Q_n(t,q):=(D+UDU)^n 1, \qquad R_n(t,q):= (D+DUU)^n 1.
\end{equation}
\end{defi}
\noindent
Several of the initial polynomials are listed below:
\begin{align*}
Q_0(t,q) &= 1 \\
Q_1(t,q) &= t \\
Q_2(t,q) &= 1+(1+q)t^2 \\
Q_3(t,q) &= (2+2q+q^2)t+(1+2q+2q^2+q^3)t^3, \\
& \\
R_0(t,q) &= 1 \\
R_1(t,q) &= (1+q)t \\
R_2(t,q) &= (1+q)+(1+2q+2q^2+q^3)t^2 \\
R_3(t,q) &= (2+5q+5q^2+3q^3+q^4)t+(1+3q+5q^2+6q^3+5q^4+3q^5+q^6)t^3.
\end{align*}

Applying the Matrix Ansatz approach, Josuat-Verg\`{e}s obtain the generating function of $Q_n(t,q)$ and $R_n(t,q)$.
\begin{thm} {\rm\bf (Josuat-Verg\`es)} \label{thm:QR-continued-fraction}
\begin{equation}\label{eqn:QR-conti}
\sum_{n\ge 0} Q_n(t,q) x^n = \mathfrak{F}(\mu^{Q}_h,\lambda^{Q}_h), \qquad
\sum_{n\ge 0} R_n(t,q) x^n = \mathfrak{F}(\mu^{R}_h,\lambda^{R}_h),
\end{equation}
where
\[
\left\{\begin{array}{rl}
\mu^{Q}_h &= tq^h([h]_q+[h+1]_q)\\ [0.8ex]
\lambda^{Q}_h &= (1+t^2q^{2h-1})[h]^2_q,
       \end{array}
\right.
\qquad
\left\{\begin{array}{rl}
\mu^{R}_h &= tq^h(1+q)[h+1]_q\\ [0.8ex]
\lambda^{R}_h &= (1+t^2q^{2h})[h]_q[h+1]_q.
       \end{array}
\right.
\]
\end{thm}
Notice that $Q_{2n}(0,q)=E_{2n}(q)$ and $R_{2n+1}(0,q)=E_{2n+1}(q)$, the $q$-secant and $q$-tangent numbers defined in Eqs.\,(\ref{eqn:q-secant}) and (\ref{eqn:q-tangent}).

\chapter{Weighted Motzkin paths and Signed Permutations}
  In this chapter we first review Flajolet's fundamental lemma on continued fractions. It helps us associate combinatorial objects whose generating functions possessing continued fraction expansion with certain weighted Motzkin paths. Then we present the weight schemes of Motzkin paths associated with the enumerator of signed permutations and the $(t,q)$-derivative polynomials respectively.

\section{Continued fractions and weighted Motzkin paths}

A \emph{Motzkin path} of length $n$ is a lattice path from the origin to the point $(n,0)$ staying weakly above the $x$-axis, using the \emph{up step} $(1,1)$, \emph{down step} $(1,-1)$, and \emph{level step} $(1,0)$.  Let $\U$, $\D$ and $\L$ denote an up step, a down step and a level step, accordingly.

We consider a Motzkin path $\mu=w_1w_2\cdots w_n$ with a weight function $\rho$ on the steps. The \emph{weight} of $\mu$, denoted by $\rho(\mu)$, is defined to be the product of the weight $\rho(w_j)$ of each step $w_j$ for $j=1,2,\dots,n$.
The \emph{height} of a step $w_j$ is the $y$-coordinate of the starting point of $w_j$.
Making use of Flajolet's formula \cite[Proposition 7A]{Flaj}, the generating function for the weight count of the Motzkin paths can be expressed as a continued fraction.

\begin{thm} {\rm\bf (Flajolet)} \label{thm:Flajolet} For $h\ge 0$, let $a_h$, $b_h$ and $c_h$ be polynomials such that each monomial has coefficient 1. Let $M_n$ be the set of weighted Motzkin paths of length $n$ such that the weight of an up step  (down step or level step, respectively) at height $h$ is one of the monomials appearing in $a_h$ ($b_h$ or $c_h$, respectively). Then the generating function for $\rho(M_n)=\sum_{\mu\in M_n} \rho(\mu)$ has the expansion
\begin{equation} \label{eqn:cf-Motzkin}
\sum_{n\ge 0} \rho(M_n)x^n
      =\frac{1}{1-c_0x}{{}\atop{-}}\frac{a_0b_1x^2}{1-c_1x}{{}\atop{-}}\frac{a_1b_2x^2}{1-c_2x}{{}\atop{-}}\cdots
\end{equation}
\end{thm}

\section{Signed permutations via bicolored Motzkin paths}
A \emph{bicolored Motzkin path} (also known as \emph{2-Motzkin path}) is a Motzkin path with two kinds of level steps, say \emph{straight} and \emph{wavy}, denoted by $\L$ and $\W$, respectively. For a nonnegative integer $h$, let $z^{(h)}$ denote a step $z$ at height $h$ in a bicolored Motzkin path for $z\in\{\U,\L,\W,\D\}$.

With theorem \ref{thm:Flajolet}, we may observe that (\ref{eqn:B-enumerator-cf}) and (\ref{eqn:QR-conti}) both are the generating functions of the weight of some weighted bicolored Motzkin paths.\\

By theorem \ref{thm:B-enumerator-cf}, the initial part of the expansion of generating function of $B_n(y,t,q)$ is
{\scriptsize
\begin{align*}
&\sum_{n\ge 0}B_n(y,t,q)x^n=\\
&\frac{1}{1-(y^2+yt)[1]_qx}{{}\atop{-}}
\frac{(1+ytq)[1]_q(y^2+yt)[1]_qx^2}{1-\left((y^2+ytq)[2]_q+(1+ytq)[1]_q\right)x}{{}\atop{-}}
\frac{(1+ytq^2)[2]_q(y^2+ytq)[2]_qx^2}{1-\left((y^2+ytq^2)[3]_q+(1+ytq^2)[2]_q\right)x}{{}\atop{\ldots}}
\end{align*}
}

Therefore, by theorem \ref{thm:Flajolet} the following set $\M_n$ of weighted bicolored Motzkin paths has the generating function of weights equal to $B_n(y,t,q)$.

\begin{defi} {\rm
Let $\M_n$ be the set of weighted bicolored Motzkin paths of length $n$ containing no wavy level steps on the $x$-axis, with a weight function $\rho$ such that for $h\ge 0$,
\begin{itemize}
\item $\rho(\U^{(h)})\in\{y^2,y^2q,\dots,y^2q^{h}\}\cup\{ytq^{h},ytq^{h+1},\dots,ytq^{2h}\}$,
\item $\rho(\L^{(h)})\in\{y^2,y^2q,\dots,y^2q^{h}\}\cup\{ytq^{h},ytq^{h+1},\dots,ytq^{2h}\}$,
\item $\rho(\W^{(h)})\in\{1,q,\dots,q^{h-1}\}\cup\{ytq^{h},ytq^{h+1},\dots,ytq^{2h-1}\}$ for $h\ge 1$,
\item $\rho(\D^{(h+1)})\in\{1,q,\dots,q^{h}\}\cup\{ytq^{h+1},ytq^{h+2},\dots,ytq^{2h+1}\}$.
\end{itemize}
}
\end{defi}

In fact, extended from Foata--Zeilberger bijection \cite{Foata-Zeil}, Corteel, Josuat-Verg\`{e}s and Kim established a bijection between the set $B_n$ of signed permutations and the set $\M_n$ of weighted bicolored Motzkin paths \cite[Subsection 7.1]{CJK}.
\begin{rem} {\rm
Rather than the original weight function given in \cite{CJK}, we have interchanged the possible weights of the up steps and the down steps, in the sense of traversing the paths backward. This unifies the possible weights for the initial step (either $\U^{(0)}$ or $\L^{(0)}$), for the purpose of restructuring the weighted bicolored Motzkin paths (Proposition \ref{pro:restructure}).
}
\end{rem}

\begin{thm}  \label{thm:Corteel} {\rm\bf (Corteel, Josuat-Verg\`es, Kim)}  There is a bijection $\Gamma$ between $B_n$ and $\M_n$ such that
\begin{equation}
\sum_{\sigma\in B_n} y^{\fwex(\sigma)}t^{\nega(\sigma)}q^{\cro_B(\sigma)} =\rho(\M_n).
\end{equation}
\end{thm}

\section{Weight Schemes for $Q_n(t,q)$ and $R_n(t,q)$}

By Theorem \ref{thm:QR-continued-fraction}, the initial part of the expansion of the generating functions for $R_n(t,q)$ and $Q_n(t,q)$ are shown below.
\begin{align*}
\sum_{n\ge 0} R_n(t,q)x^n &=\frac{1}{1-t(1+q)[1]_qx}{{}\atop{-}}\frac{(1+t^2q^2)[1]_q[2]_qx^2}{1-tq(1+q)[2]_qx}{{}\atop{-}}\frac{(1+t^2q^4)[2]_q[3]_qx^2}{1-tq^2(1+q)[3]_qx}\cdots, \\
\sum_{n\ge 0} Q_n(t,q)x^n &=\frac{1}{1-t[1]_qx}{{}\atop{-}}\frac{(1+t^2q)[1]_q^2x^2}{1-tq([1]_q+[2]_q)x}{{}\atop{-}}\frac{(1+t^2q^3)[2]_q^2x^2}{1-tq^2([2]_q+[3]_q)x}\cdots.
\end{align*}

Therefore, by Theorem \ref{thm:Flajolet} the following observations provide feasible weight schemes that realize the polynomials  $R_n(t,q)$ and $Q_n(t,q)$ in terms of the bicolored Motzkin paths. 

\begin{pro} \label{pro:Tn-weighting}
Let $\T_n$ be the set of weighted bicolored Motzkin paths of length $n$ with a weight function $\rho$ such that for $h\ge 0$,
\begin{itemize}
\item $\rho(\U^{(h)})\in\{1,q,\dots,q^{h}\}\cup\{t^2q^{2h+2},t^2q^{2h+3},\dots,t^2q^{3h+2}\}$,
\item $\rho(\L^{(h)})\in\{tq^{h+1},tq^{h+2},\dots,tq^{2h+1}\}$,
\item $\rho(\W^{(h)})\in\{tq^{h},tq^{h+1},\dots,tq^{2h}\}$,
\item $\rho(\D^{(h+1)})\in\{1,q,\dots,q^{h+1}\}$.
\end{itemize}
Then we have
\[
\sum_{n\ge 0} \rho(\T_n)x^n = \sum_{n\ge 0} R_n(t,q)x^n.
\]
\end{pro}

\begin{pro} \label{pro:T*n-weighting}
Let $\T^*_n$ be the set of  weighted bicolored Motzkin paths of length $n$ containing no wavy level steps on the $x$-axis, with a weight function $\rho$ such that for $h\ge 0$,
\begin{itemize}
\item $\rho(\U^{(h)})\in\{1,q,\dots,q^{h}\}\cup\{t^2q^{2h+1},t^2q^{2h+2},\dots,t^2q^{3h+1}\}$,
\item $\rho(\L^{(h)})\in\{tq^{h},tq^{h+1},\dots,tq^{2h}\}$,
\item $\rho(\W^{(h)})\in\{tq^{h},tq^{h+1},\dots,tq^{2h-1}\}$ for $h\ge 1$,
\item $\rho(\D^{(h+1)})\in\{1,q,\dots,q^{h}\}$.
\end{itemize}
Then we have
\[
\sum_{n\ge 0} \rho(\T^{*}_n)x^n = \sum_{n\ge 0} Q_n(t,q)x^n.
\]
\end{pro}

\noindent
Our second main result is, as Corteel et al. connecting $\M_n$ with $B_n$, to use these set $\T_n$ and $\T^*_n$ of  paths to encode the snakes in $\SS^{00}_{n+1}$ and $\SS^{0}_n$ (Theorem \ref{thm:T->snakes-00} and Theorem \ref{thm:T*->snake-0})

\chapter{Signed Countings on type B and D}

  The results in chapter 4 and 5 are the joint works with Sen-Peng Eu, Tung-Shan Fu and Hsiang-Chun Hsu, part of them had appeared in \cite{EFHL2018}. 
  
  In this chapter we present our type B and D extension of signed counting results. To prove the result, we first apply Corteel et al.'s bijection between $B_n$ and $\M_n$ mentioned in chapter 3 with some adjusts. Then we construct an involution on weighted Motzkin paths which encode the signed counting process combinatorially. Then we compared the set of fixed points of the involutions with the set of paths associated with $Q_n(t,q)$ and $R_n(t,q)$.
\section{Signed Countings on type B and D}

 Our first main result is the type B and type D analogs of
 Eqs.\,(\ref{JV1}) and (\ref{JV2}), with the sign of $\sigma\in B_n$ depending on the parity of one half of the statistic $\fwex(\sigma)$. Amazingly, the signed counting turns out to be related to the derivative polynomials $Q_n(t,q)$ and $R_n(t,q)$.

\begin{thm} \label{thm:signFwexB}
For $n\ge 1$, we have
\begin{enumerate}
\item  ${\displaystyle
\sum_{\sigma\in B_n}(-1)^{\lfloor \frac{\fwex(\sigma)}{2}\rfloor} t^{\nega(\sigma)}q^{\cro_B(\sigma)}
=\begin{cases}
	(-1)^{\frac{n}{2}}(t+1) R_{n-1}(t,q) &\mbox{, if $n$ is odd;}\\
	(-1)^{\frac{n-1}{2}}(t-1) R_{n-1}(t,q) &\mbox{, if $n$ is even.}
\end{cases}
}$
\item ${\displaystyle
\sum_{\sigma\in B_n}(-1)^{\lceil \frac{\fwex(\sigma)}{2}\rceil} t^{\nega(\sigma)}q^{\cro_B(\sigma)}=\begin{cases}
	(-1)^{\frac{n}{2}}(t-1)R_{n-1}(t,q) & \mbox{ if $n$ is even;}\\
	(-1)^{\frac{n+1}{2}}(t+1)R_{n-1}(t,q) & \mbox{ if $n$ is odd.} 
			\end{cases}	.
}$
\end{enumerate}
\end{thm}

\begin{cor} \label{thm:signFwexD}
For $n\ge 1$, we have
\begin{align*}
&\sum_{\sigma\in D_n}(-1)^{\lfloor\frac{\fwex(\sigma)}{2}\rfloor} t^{\nega(\sigma)}q^{\cro_B(\sigma)}
     =\sum_{\sigma\in D_n}(-1)^{\lceil\frac{\fwex(\sigma)}{2}\rceil} t^{\nega(\sigma)}q^{\cro_B(\sigma)}\\
    =&\left\{\begin{array}{ll} (-1)^\frac{n}{2}tR_{n-1}(t,q)  & \mbox{ if $n$ is even,}\\
	                           (-1)^\frac{n+1}{2}R_{n-1}(t,q) & \mbox{ if $n$ is odd.}
             \end{array}
      \right.
\end{align*}
\end{cor}

\begin{thm} \label{thm:signFwexB*}
For $n\ge 1$, we have
\begin{enumerate}
\item ${\displaystyle
\sum_{\sigma\in B_n^{*}} \left(-\frac{1}{q}\right)^{\lfloor\frac{\fwex(\sigma)}{2}\rfloor} t^{\nega(\sigma)}q^{\cro_B(\sigma)}
=\left(-\frac{1}{q}\right)^{\lfloor\frac{n}{2}\rfloor}Q_n(t,q).
}$
\item ${\displaystyle
\sum_{\sigma\in B_n^{*}} \left(-\frac{1}{q}\right)^{\lceil\frac{\fwex(\sigma)}{2}\rceil} t^{\nega(\sigma)}q^{\cro_B(\sigma)}
=\left(-\frac{1}{q}\right)^{\lceil\frac{n}{2}\rceil}Q_n(t,q).
}$
\end{enumerate}
\end{thm}

\begin{cor} \label{thm: signFwexD*}
For $n\ge 1$, we have
\begin{align*}
 &\sum_{\sigma\in D_n^{*}} \left(-\frac{1}{q}\right)^{\lfloor\frac{\fwex(\sigma)}{2}\rfloor} t^{\nega(\sigma)}q^{\cro_B(\sigma)}
     =\sum_{\sigma\in D_n^{*}} \left(-\frac{1}{q}\right)^{\lceil\frac{\fwex(\sigma)}{2}\rceil} t^{\nega(\sigma)}q^{\cro_B(\sigma)}\\
     =&\left\{\begin{array}{ll} \big(-\frac{1}{q}\big)^\frac{n}{2}Q_{n}(t,q) & \mbox{ if $n$ is even,}\\
	                            0 & \mbox{ if $n$ is odd.}
             \end{array}
      \right.
\end{align*}
\end{cor}
\noindent
Setting $t=1$ and $q=1$, we obtain types B and D extensions of the results in Eqs.\,(\ref{Eulercan1}) and (\ref{Eulercan2}).

\begin{cor} \label{cor:(-1)-eval-Bn}
For $n\ge 1$, we have
\begin{enumerate}
\item ${\displaystyle
\sum_{\sigma\in B_n}(-1)^{\lfloor\frac{\fwex(\sigma)}{2}\rfloor}
      =\left\{\begin{array}{ll} (-1)^\frac{n}{2}2^nE_n & \mbox{ if $n$ is even,}\\
	                             0 &  \mbox{ if $n$ is odd.}
              \end{array}
       \right.
}$
\item ${\displaystyle
\sum_{\sigma\in B_n}(-1)^{\lceil \frac{\fwex(\sigma)}{2}\rceil} =\begin{cases}
		0 						   & \mbox{ if $n$ is even;}\\
		(-1)^{\frac{n+1}{2}}2^nE_n & \mbox{ if $n$ is odd.} 
	\end{cases}
}$
\item ${\displaystyle
\sum_{\sigma\in D_n}(-1)^{\lfloor\frac{\fwex(\sigma)}{2}\rfloor}=\sum_{\sigma\in D_n}(-1)^{\lceil\frac{\fwex(\sigma)}{2}\rceil}=(-1)^{\lfloor\frac{n+1}{2}\rfloor}2^{n-1} E_n.
}$
\end{enumerate}
\end{cor}

\begin{cor} \label{cor:(-1)-eval-Bn*}
For $n\ge 1$, we have
\begin{enumerate}
\item ${\displaystyle
\sum_{\sigma\in B_n^{*}}(-1)^{\lfloor\frac{\fwex(\sigma)}{2}\rfloor}=(-1)^{\lfloor\frac{n}{2}\rfloor} S_n.
}$
\item ${\displaystyle
\sum_{\sigma\in B_n^*}(-1)^{\lceil \frac{\fwex(\sigma)}{2}\rceil}=(-1)^{\lceil\frac{n}{2}\rceil}S_n}$.
\item ${\displaystyle
\sum_{\sigma\in D_n^{*}}(-1)^{\lfloor\frac{\fwex(\sigma)}{2}\rfloor}
	  =\sum_{\sigma\in D_n^{*}}(-1)^{\lceil\frac{\fwex(\sigma)}{2}\rceil}
      =\left\{\begin{array}{ll} (-1)^\frac{n}{2} S_n & \mbox{ if $n$ is even,}\\
	                            0 & \mbox{ if $n$ is odd.}
              \end{array}
      \right.
}$
\end{enumerate}
\end{cor}

Subtracting Corollary \ref{cor:(-1)-eval-Bn}(i) with \ref{cor:(-1)-eval-Bn*}{i} and Corollary \ref{cor:(-1)-eval-Bn}(ii) with Corollary \ref{cor:(-1)-eval-Bn*}(ii), and applying Hoffman's result $P_n(1)-Q_n(1)=S_n^D$, we obtain the following identities of Springer numbers of type D.
\begin{cor} \label{cor:sign-S_n^D}
For $n\ge 1$, we have
\begin{enumerate}
\item ${\displaystyle
\sum_{\sigma\in B_n-B_n^{*}}(-1)^{\lfloor\frac{\fwex(\sigma)}{2}\rfloor}
	=\left\{\begin{array}{ll} 
		(-1)^\frac{n}{2}S_n^D 		&  \mbox{ if $n$ is even,}\\
	    (-1)^{\frac{n+1}{2}}S_n		&  \mbox{ if $n$ is odd.}
              \end{array}
       \right.
}$
\item ${\displaystyle
\sum_{\sigma\in B_n-B_n^{*}}(-1)^{\lceil\frac{\fwex(\sigma)}{2}\rceil}
	=\left\{\begin{array}{ll} 
		(-1)^{\frac{n}{2}+1}S_n 		&  \mbox{ if $n$ is even,}\\
	    (-1)^{\frac{n+1}{2}}S_n^D		&  \mbox{ if $n$ is odd.}
              \end{array}
       \right.
}$
\end{enumerate}
\end{cor}

\section{The cases of $B_n$ and $D_n$}

In this section we present a combinatorial proof of Theorem \ref{thm:signFwexB} and Corollary \ref{thm:signFwexD}, via a sign-reversing involution on corresponding set of paths.
First, notice that plugging in $y=\sqrt{-1}$ in $B_n(y,t,q)$, we obtain
\begin{align*}
	B_n(\sqrt{-1},t,q)
   &=\sum_{\substack{\sigma\in B_n \\ 2|\fwex(\sigma)}}(\sqrt{-1})^{\fwex(\sigma)} t^{\nega(\sigma)}q^{\cro_B(\sigma)}+\sum_{\substack{\sigma\in B_n \\ 2\nmid \fwex(\sigma)}}(\sqrt{-1})^{\fwex(\sigma)} t^{\nega(\sigma)}q^{\cro_B(\sigma)}\\
 &=\sum_{\substack{\sigma\in B_n \\ 2|\fwex(\sigma)}}(-1)^{\frac{\fwex(\sigma)}{2}} t^{\nega(\sigma)}q^{\cro_B(\sigma)}+\sqrt{-1}\sum_{\substack{\sigma\in B_n \\ 2\nmid \fwex(\sigma)}}(-1)^{\frac{\fwex(\sigma)-1}{2}} t^{\nega(\sigma)}q^{\cro_B(\sigma)}.
\end{align*}
Then it is easy to see that 
\[
	\sum_{\sigma\in B_n}(-1)^{\lfloor\frac{\fwex(\sigma)}{2}\rfloor}t^{\nega(\sigma)}q^{\croB(\sigma)}
	={\rm Re}(B_n(\sqrt{-1},t,q))+{\rm Im}(B_n(\sqrt{-1},t,q))
\]
and
\[
	\sum_{\sigma\in B_n}(-1)^{\lceil\frac{\fwex(\sigma)}{2}\rceil}t^{\nega(\sigma)}q^{\croB(\sigma)}
	={\rm Re}(B_n(\sqrt{-1},t,q))-{\rm Im}(B_n(\sqrt{-1},t,q)).
\]
Since $B_n(y,t,q)$ is the generating function of weights of paths in $\M_n$, we want to do the sign counting combinatorially via these paths. In order to do so, we restructure the weighted bicolored Motzkin paths in $\M_n$. 
Recall the weight scheme of $\M_n$ is the following:
\begin{itemize}
\item $\rho(\U^{(h)})\in\{y^2,y^2q,\dots,y^2q^{h}\}\cup\{ytq^{h},ytq^{h+1},\dots,ytq^{2h}\}$,
\item $\rho(\L^{(h)})\in\{y^2,y^2q,\dots,y^2q^{h}\}\cup\{ytq^{h},ytq^{h+1},\dots,ytq^{2h}\}$,
\item $\rho(\W^{(h)})\in\{1,q,\dots,q^{h-1}\}\cup\{ytq^{h},ytq^{h+1},\dots,ytq^{2h-1}\}$ for $h\ge 1$,
\item $\rho(\D^{(h+1)})\in\{1,q,\dots,q^{h}\}\cup\{ytq^{h+1},ytq^{h+2},\dots,ytq^{2h+1}\}$.
\end{itemize}

\begin{defi} {\rm
	Let $\H_n$ be the set of weighted bicolored Motzkin paths of length $n$ with a weight function $\rho$ such that for $h\ge 0$,
\begin{itemize}
\item $\rho(\U^{(h)})\in\{y^2,y^2q,\dots,y^2q^{h+1}\}\cup\{ytq^{h+1},ytq^{h+2},\dots,ytq^{2h+2}\}$,
\item $\rho(\L^{(h)})\in\{1,q,\dots,q^{h}\}\cup\{ytq^{h+1},ytq^{h+2},\dots,ytq^{2h+1}\}$,
\item $\rho(\W^{(h)})\in\{y^2,y^2q,\dots,y^2q^{h}\}\cup\{ytq^{h},ytq^{h+1},\dots,ytq^{2h}\}$,
\item $\rho(\D^{(h+1)})\in\{1,q,\dots,q^{h}\}\cup\{ytq^{h+1},ytq^{h+2},\dots,ytq^{2h+1}\}$.
\end{itemize}
}
\end{defi}	
	
\begin{pro} \label{pro:restructure}
	There is a two-to-one bijection $\Phi$ between $\M_n$ and $\H_{n-1}$ such that
	\[
	\rho(\M_n)=(y^2+yt)\rho(\H_{n-1}).
	\]
\end{pro}

\begin{proof} 
Given a path $\mu=p_1p_2\cdots p_n\in\M_n$, we create a weight-preserving path $z_1z_2\cdots z_{2n}$ of length $2n$ from $\mu$ as the intermediate stage, where $z_{2i-1}z_{2i}$ is determined from $p_i$ by
\[
z_{2i-1}z_{2i}=\left\{\begin{array}{ll}
                   \U\U & \mbox{ if $p_i=\U$,}\\
	               \U\D & \mbox{ if $p_i=\L$,}\\
                   \D\U & \mbox{ if $p_i=\W$,}\\
                   \D\D & \mbox{ if $p_i=\D$,}
              \end{array}
      \right.
\]
with weight $\rho(z_{2i-1})=\rho(p_i)$ and $\rho(z_{2i})=1$ for $1\le i\le n$. Note that $\rho(z_1)=\rho(p_1)\in\{y^2,yt\}$ since $p_1=\U^{(0)}$ or $\L^{(0)}$. Then we construct the corresponding path $\Phi(\mu)=p'_1p'_2\cdots p'_{n-1}$ from $z_2\cdots z_{2n-1}$ (with $z_1$ and $z_{2n}$ excluded), where $p'_j$ is determined by $z_{2j}z_{2j+1}$ with weight $\rho(p'_j)=\rho(z_{2j})\rho(z_{2j+1})$  according to the following cases
\[
p'_j=\left\{\begin{array}{ll}
                   \U & \mbox{ if $z_{2j}z_{2j+1}=\U\U$,}\\
	               \L & \mbox{ if $z_{2j}z_{2j+1}=\U\D$,}\\
                   \W & \mbox{ if $z_{2j}z_{2j+1}=\D\U$,}\\
                   \D & \mbox{ if $z_{2j}z_{2j+1}=\D\D$,}
              \end{array}
      \right.
\]
for $1\le j\le n-1$. Notice that $\rho(\mu)=\rho(p_1)\rho(\Phi(\mu))$ since $\rho(p'_j)=\rho(p_{j+1})$ for $1\le j\le n-1$. Hence $\rho(\mu)=y^2\rho(\Phi(\mu))$ or $yt\rho(\Phi(\mu))$. Moreover, the possible weights of $p'_j$ can be determined from the steps of $\mu$ since
\[
	p'_j=\left\{\begin{array}{ll}
	\U^{(h)} & \mbox{ if $p_{j+1}=\U^{(h+1)}$ or $\L^{(h+1)}$,}\\
	\L^{(h)} & \mbox{ if $p_{j+1}=\W^{(h+1)}$ or $\D^{(h+1)}$,}\\
	\W^{(h)} & \mbox{ if $p_{j+1}=\U^{(h)}$ or $\L^{(h)}$,}\\
	\D^{(h+1)} & \mbox{ if $p_{j+1}=\D^{(h+1)}$ or $\W^{(h+1)}$,}
	\end{array}
	\right.
\]
for some $h\ge 0$. That $\Phi(\mu)\in\H_{n-1}$ follows from the weight function of the paths in $\M_n$.

It is straightforward to obtain the map $\Phi^{-1}$ by the reverse procedure.
\end{proof}

\begin{exa} {\rm See the figure below. Let $\mu=p_1p_2\dots p_8\in\M_8$ be the path shown on the left-hand side in the upper row, where $w_j=\rho(p_j)$ for $1\le j\le 8$. The corresponding  bicolored Motzkin path $\Phi(\mu)=p'_1p'_2\cdots p'_7\in\H_7$  is shown on the right-hand side, where the intermediate stage $z_1z_2\cdots z_{16}$ is shown in the lower row.
}
\begin{center}
\psfrag{1}[][][0.85]{$1$}
\psfrag{w1}[][][0.85]{$w_1$}
\psfrag{w2}[][][0.85]{$w_2$}
\psfrag{w3}[][][0.85]{$w_3$}
\psfrag{w4}[][][0.85]{$w_4$}
\psfrag{w5}[][][0.85]{$w_5$}
\psfrag{w6}[][][0.85]{$w_6$}
\psfrag{w7}[][][0.85]{$w_7$}
\psfrag{w8}[][][0.85]{$w_8$}
\psfrag{Phi}[][][0.85]{$\Phi$}
\includegraphics[width=4.4in]{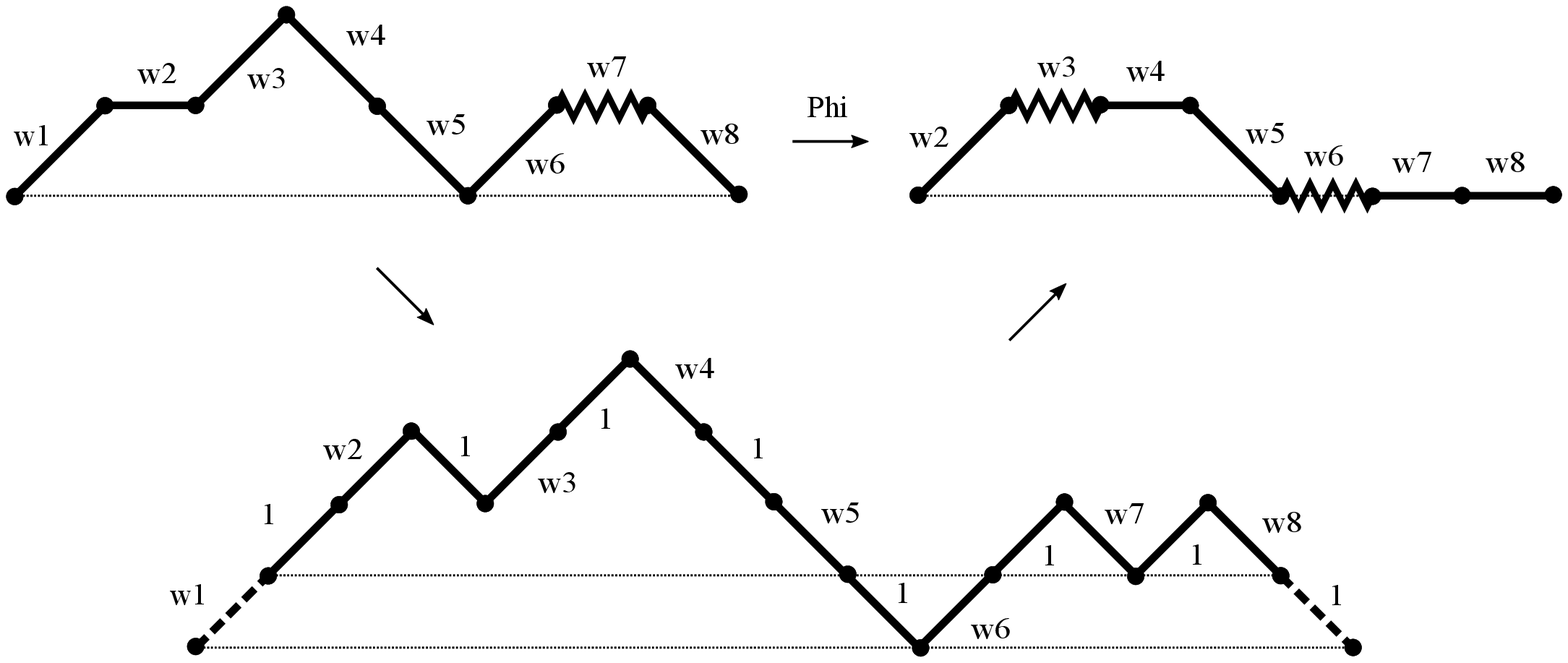}
\end{center}
\end{exa}

	By a \emph{matching} pair $(\U^{(h)},\D^{(h+1)})$ we mean an up step $\U^{(h)}$ and a down step $\D^{(h+1)}$ that face each other, in the sense that the horizontal line segment from the midpoint of $\U^{(h)}$ to the midpoint of $\D^{(h+1)}$ stays under the path.
	
	Let $\rho(\U^{(h)},\D^{(h+1)})$ denote the ordered pair $(\rho(\U^{(h)}),\rho(\D^{(h+1)}))$ of weights. We shall establish an involution $\Psi_1:\H_n\rightarrow\H_n$ that changes the weight of a path by the factor $y^2$, with the following set of restricted paths as the fixed points.

\begin{defi} \label{def:Fn} {\rm
Let $\F_n\subset\H_n$ be the subset consisting of the weighted paths satisfying the following conditions. For $h\ge 0$,
\begin{itemize}
\item $\rho(\U^{(h)},\D^{(h+1)})=(y^2q^a,q^b)$ or $(ytq^{h+1+a},ytq^{h+1+b})$ for some $a\in\{0,1,\dots, h+1\}$ and $b\in\{0,1,\dots, h\}$, for any matching pair $(\U^{(h)},\D^{(h+1)})$,
	\item $\rho(\L^{(h)})\in\{ytq^{h+1}, ytq^{h+2},\dots,ytq^{2h+1} \}$,
	\item $\rho(\W^{(h)})\in\{ytq^{h},ytq^{h+1},\dots,ytq^{2h}\}$.
\end{itemize}
}
\end{defi}

Notice that for any matching pair $(\U^{(h)},\D^{(h+1)})$ with weight $(a,b)$, it is equivalent to the reassignment $\rho(\U^{(h)},\D^{(h+1)})=(a',b')$ such that $a'b'=ab$ regarding the total weight of a path.
Comparing the weight functions of the paths in $\F_n$ and in the set $\T_n$ given in Proposition \ref{pro:Tn-weighting}, we have the following result.
	
\begin{lem} \label{lem:Fn=y^nRn} We have
\[
\rho(\F_n)=y^n R_n(t,q).
\]
\end{lem}

\begin{proof} For any path $\mu\in\F_n$, notice that the weight of $\mu$ contains the factor $y^n$ since every matching pair of up step and down step contributes the parameter $y^2$, while every level step (either straight or wavy) contributes the parameter $y$.
By Theorem \ref{thm:Flajolet}, we observe that $\rho(\F_n)$ and $y^n\cdot \rho(\T_n)$ have same generating function. By Proposition \ref{pro:Tn-weighting}, the assertion follows.
\end{proof}

\smallskip
The `sign-reversing' map $\Psi_1:\H_n\rightarrow\H_n$ is constructed as follows.

\smallskip
\noindent
{\bf Algorithm A.}

Given a path $\mu\in\H_n$, the corresponding path $\Psi_1(\mu)$ is constructed by the following procedure.	
\begin{enumerate}
	\item[(A1)] If $\mu$ contains no straight level steps $\L^{(h)}$ with weight $q^a$ or wavy level steps $\W^{(h)}$ with weight $y^2q^a$ for any $a\in\{0,1,\dots,h\}$ then go to (A2). Otherwise, among such level steps find the first step $z$, say $z=\L^{(h)}$ ($\W^{(h)}$, respectively) with weight $\rho(z)=q^a$ ($y^2q^a$, respectively), then the corresponding path $\Psi_1(\mu)$ is obtained by replacing $z$ by $\W^{(h)}$ ($\L^{(h)}$, respectively) with weight $y^2q^a$ ($q^a$, respectively).
	\item[(A2)] If $\mu$ contains no matching pairs $(\U^{(h)},\D^{(h+1)})$ with $\rho(\U^{(h)},\D^{(h+1)})=(y^2q^a, ytq^{h+1+b})$ or $(ytq^{h+1+a}, q^{b})$ for any $a\in\{0,1,\dots, h+1\}$ or  $b\in\{0,1,\dots h\}$ then go to (A3). Otherwise, among such pairs find the first pair with weight, say $(y^2q^a, ytq^{h+1+b})$ ($(ytq^{h+1+a}, q^{b})$, respectively), then the corresponding path $\Psi_1(\mu)$ is obtained by replacing the weight of the pair by $(ytq^{h+1+a}, q^{b})$ ($(y^2q^a, ytq^{h+1+b})$, respectively).
	\item[(A3)] Then we have $\mu\in\F_n$. Let $\Psi_1(\mu)=\mu$.
	\end{enumerate}

\smallskip
Regarding the possibilities of the weighted steps of the paths in $\H_n$, we have the following immediate result.
	
\begin{pro} \label{pro:map-psi1} The map $\Psi_1$ established by Algorithm A is an involution on the set $\H_n$ such that for any path $\mu\in\H_n$,  $\Psi_1(\mu)=\mu$ if $\mu\in\F_n$, and $\rho(\Psi_1(\mu))=y^2\rho(\mu)$ or $y^{-2}\rho(\mu)$ otherwise.
\end{pro}	

Now we are ready to prove Theorem \ref{thm:signFwexB}.

\smallskip
\noindent
\emph{Proof of Theorem \ref{thm:signFwexB}.} For (i), by Theorem \ref{thm:Corteel} and Proposition \ref{pro:restructure}, we have
\[
\sum_{\sigma\in B_n} y^{\fwex(\sigma)}t^{\nega(\sigma)}q^{\cro_B(\sigma)} =(y^2+yt)\rho(\H_{n-1}).
\]
Then the expression $\sum_{\sigma\in B_n}(-1)^{\lfloor \frac{\fwex(\sigma)}{2}\rfloor} t^{\nega(\sigma)}q^{\cro_B(\sigma)}$ equals the sum of the real part and the imaginary part of the polynomial $(y^2+yt)\rho(\H_{n-1})$ evaluated at $y=\sqrt{-1}$. By  Proposition \ref{pro:map-psi1} and Lemma \ref{lem:Fn=y^nRn}, we have
\begin{align*}
(y^2+yt)\rho(\H_{n-1})\big|_{y=\sqrt{-1}}
            &= (y^2+yt)\rho(\F_{n-1})\big|_{y=\sqrt{-1}}  \\
            &= (y^{n+1}+y^nt) R_{n-1}(t,q)\big|_{y=\sqrt{-1}}\\
            &=\begin{cases}
            	\left[(-1)^{\frac{n}{2}}\sqrt{-1}+(-1)^{\frac{n}{2}}t\right]R_n(t,q) & \mbox{, if $n$ is even;}\\
            	\left[(-1)^{\frac{n+1}{2}}+(-1)^{\frac{n-1}{2}}t\sqrt{-1}\right]R_{n-1}(t,q) & \mbox{, if $n$ is odd.}
            \end{cases}
\end{align*}
Taking the real part and the imaginary part of the above evaluation leads to
\[
\sum_{\sigma\in B_n}(-1)^{\lfloor \frac{\fwex(\sigma)}{2}\rfloor} t^{\nega(\sigma)}q^{\cro_B(\sigma)}
	=\begin{cases}
		(-1)^{\frac{n}{2}}(t+1) R_{n-1}(t,q) &\mbox{, if $n$ is odd;}\\
		(-1)^{\frac{n-1}{2}}(t-1) R_{n-1}(t,q) &\mbox{, if $n$ is even.}
	\end{cases}
\] 	
as required.\\
For (ii), similarly the expression $\sum_{\sigma\in B_n}(-1)^{\lceil \frac{\fwex(\sigma)}{2}\rceil} t^{\nega(\sigma)}q^{\cro_B(\sigma)}$ equals the real part subtracting the imaginary part of the polynomial $(y^2+yt)\rho(\H_{n-1})$ evaluated at $y=\sqrt{-1}$. Therefore, we have
\[
	\sum_{\sigma\in D_n}(-1)^{\lfloor\frac{\fwex(\sigma)}{2}\rfloor} t^{\nega(\sigma)}q^{\cro_B(\sigma)}
     =\left\{\begin{array}{ll} 
     (-1)^\frac{n}{2}tR_{n-1}(t,q)  & \mbox{ if $n$ is even,}\\
     (-1)^\frac{n+1}{2}R_{n-1}(t,q) & \mbox{ if $n$ is odd.}
             \end{array}
      \right.
\] \qed

\smallskip
In the following we shall prove Corollary \ref{thm:signFwexD}. Recall that the set $D_n$ of even-signed permutations consists of the signed permutations with even number of negative entries.

\begin{defi} {\rm
Let $\M'_n\subset \M_n$ be the subset consisting of the  paths whose weights contain even powers of $t$. Let $\H^{(1)}_n$ ($\H^{(2)}_n$, respectively) be the subset of $\H_n$ consisting of the paths whose weights contain odd (even, respectively) powers of $t$.
}
\end{defi}

Notice that the bijection $\Gamma: B_n\rightarrow\M_n$ in Theorem \ref{thm:Corteel} induces a bijection between $D_n$ and $\M'_n$ such that
\begin{equation} \label{eqn:Dn-M'n}
\sum_{\sigma\in D_n} y^{\fwex(\sigma)}t^{\nega(\sigma)}q^{\cro_B(\sigma)} = \rho(\M'_n).
\end{equation}
Moreover, the involution $\Psi_1:\H_n\rightarrow\H_n$ and the set $\F_n$ of fixed points have the following properties.

\begin{lem} \label{lem:induced-map-psi1} The map $\Psi_1$ established by Algorithm A induces an involution on $\H^{(1)}_n$ and $\H^{(2)}_n$, respectively. Moreover, for any path $\mu\in\F_n$, the power of $t$ of $\rho(\mu)$ has the same parity of $n$.
\end{lem}

\begin{proof} By Proposition \ref{pro:map-psi1}, we observe that the map $\Psi_1:\H_n\rightarrow\H_n$ preserves the powers of $t$ of the weight of the paths. By the
weight conditions of $\mu\in\F_n$ given in Definition \ref{def:Fn}, we observe that every matching pair $(\U^{(h)},\D^{(h+1)})$ contributes the parameter $t^0$ or $t^2$ to $\rho(\mu)$, while every level step contributes the parameter $t$ to $\rho(\mu)$. The assertions follow.
\end{proof}

\noindent
Now we are ready to prove Corollary \ref{thm:signFwexD}.

\smallskip
\noindent
\emph{Proof of Corollary \ref{thm:signFwexD}.} By Proposition \ref{pro:restructure} and Eq.\,(\ref{eqn:Dn-M'n}),  taking the terms with even powers of $t$ yields
\[
\sum_{\sigma\in D_n} y^{\fwex(\sigma)}t^{\nega(\sigma)}q^{\cro_B(\sigma)} =y^2\rho(\H^{(2)}_{n-1})+yt\rho(\H^{(1)}_{n-1}).
\]
Consider the polynomial $y^2\rho(\H^{(2)}_{n-1})+yt\rho(\H^{(1)}_{n-1})$ evaluated at $y=\sqrt{-1}$. By Proposition \ref{pro:map-psi1} and Lemma  \ref{lem:induced-map-psi1}, for $n$ odd, we have $\rho(\H^{(1)}_{n-1})\big|_{y=\sqrt{-1}}=0$ and
\begin{align*}
y^2\rho(\H^{(2)}_{n-1})\big|_{y=\sqrt{-1}}
            &= y^2\rho(\F_{n-1})\big|_{y=\sqrt{-1}}  \\
            &= y^{n+1} R_{n-1}(t,q)\big|_{y=\sqrt{-1}},
\end{align*}
Moreover, for $n$ even, we have $\rho(\H^{(2)}_{n-1})\big|_{y=\sqrt{-1}}=0$ and
\begin{align*}
yt\rho(\H^{(1)}_{n-1})\big|_{y=\sqrt{-1}}
            &= yt\rho(\F_{n-1})\big|_{y=\sqrt{-1}}  \\
            &= y^{n}t R_{n-1}(t,q)\big|_{y=\sqrt{-1}}.
\end{align*}
Hence we have
\[
\sum_{\sigma\in D_n}(-1)^{\lfloor \frac{\fwex(\sigma)}{2}\rfloor} t^{\nega(\sigma)}q^{\cro_B(\sigma)}=\left\{
               \begin{array}{ll}
                (-1)^{\frac{n+1}{2}} R_{n-1}(t,q) &\mbox{if $n$ is odd,} \\
                (-1)^{\frac{n}{2}}tR_{n-1}(t,q) &\mbox{if $n$ is even.}
               \end{array}
               \right.
\]
Note that we obtain the same result if we replace $(-1)^{\lfloor \frac{\fwex(\sigma)}{2}\rfloor}$ by $(-1)^{\lceil \frac{\fwex(\sigma)}{2}\rceil}$, since every $\sigma\in D_n$ has even number of negative entries, which leads to $\lfloor\frac{\fwex(\sigma)}{2}\rfloor=\lceil\frac{\fwex(\sigma)}{2}\rceil=\frac{\fwex(\sigma)}{2}$.	
The proof of Corollary \ref{thm:signFwexD} is completed.

\section{The cases of $B_n^*$ and $D_n^*$}
In this section we prove Theorem \ref{thm:signFwexB*} and Corollary \ref{thm: signFwexD*} in terms of the weighted paths associated to the set $B^{*}_n$ of signed permutations without fixed points.

Let $B_n^*(y,t,q)=\sum_{\sigma\in B_n^*}y^{\fwex(\sigma)}t^{\nega(\sigma)}q^{\croB(\sigma)}$.
Notice that plugging in $y=\sqrt{\frac{-1}{q}}$ in $B_n^*(y,t,q)$, we obtain
\begin{align*}
	&B_n^*\left(\sqrt{\frac{-1}{q}},t,q\right)\\
   =&\sum_{\substack{\sigma\in B_n^* \\ 2|\fwex(\sigma)}}\left(\sqrt{\frac{-1}{q}}\right)^{\fwex(\sigma)} t^{\nega(\sigma)}q^{\cro_B(\sigma)}+\sum_{\substack{\sigma\in B_n^* \\ 2\nmid \fwex(\sigma)}}\left(\sqrt{\frac{-1}{q}}\right)^{\fwex(\sigma)} t^{\nega(\sigma)}q^{\cro_B(\sigma)}\\
 =&\sum_{\substack{\sigma\in B_n^* \\ 2|\fwex(\sigma)}}\left(\frac{-1}{q}\right)^{\frac{\fwex(\sigma)}{2}} t^{\nega(\sigma)}q^{\cro_B(\sigma)}+\sqrt{\frac{-1}{q}}\sum_{\substack{\sigma\in B_n \\ 2\nmid \fwex(\sigma)}}\left(\frac{-1}{q}\right)^{\frac{\fwex(\sigma)-1}{2}} t^{\nega(\sigma)}q^{\cro_B(\sigma)}.
\end{align*}
It is easy to see that
\[
	\sum_{\sigma\in B_n^*}\left(\frac{-1}{q}\right)^{\lfloor\frac{\fwex(\sigma)}{2}\rfloor}t^{\nega(\sigma)}q^{\croB(\sigma)}={\rm Re}\left(B_n^*\left(\sqrt{\frac{-1}{q}},t,q\right)\right)+\sqrt{q}\cdot{\rm Im}\left(B_n^*\left(\sqrt{\frac{-1}{q}},t,q\right)\right)
\]
and
\[
	\sum_{\sigma\in B_n^*}\left(\frac{-1}{q}\right)^{\lceil\frac{\fwex(\sigma)}{2}\rceil}t^{\nega(\sigma)}q^{\croB(\sigma)}={\rm Re}\left(B_n^*\left(\sqrt{\frac{-1}{q}},t,q\right)\right)-\sqrt{q}\cdot{\rm Im}\left(B_n^*\left(\sqrt{\frac{-1}{q}},t,q\right)\right)
\]
We show that $B_n^*(y,t,q)$ is the generating function of weights of some subset $\M_n^{*}$ of $\M_n$ which is easily described, so we can do the sign counting combinatorially via $\M_n^{*}$.

By the definition of the crossings of signed permutations, for any $\sigma=\sigma_1\sigma_2\cdots\sigma_n\in B_n$ we observe that if $(i,j)$ is a crossing of $\sigma$ then $\sigma_i\neq i$ and $\sigma_j\neq j$, i.e., the fixed points of $\sigma$ are not involved in any crossing of $\sigma$. The following fact is a property of the bijection $\Gamma:B_n\rightarrow\M_n$ given in Theorem \ref{thm:Corteel}.

\begin{lem} \label{lem:no-xing} For a $\sigma=\sigma_1\sigma_2\cdots\sigma_n\in B_n$, let $\Gamma(\sigma)=z_1z_2\cdots z_n\in\M_n$ be the corresponding weighted bicolored Motzkin path. Then for $j\in [n]$, $\sigma_j=j$ if and only if
the step $z_j$ is a straight level step with weight $y^2$.
\end{lem}

Let $\M^{*}_n\subset \M_n$ be the subset consisting of the paths containing no straight level steps with weight $y^2$. It follows from Lemma \ref{lem:no-xing} that the bijection $\Gamma:B_n\rightarrow\M_n$ induces a bijection between $B^{*}_n$ and $\M^{*}_n$ such that
\begin{equation} \label{eqn:expression-B*n}
\rho(\M^{*}_n)=\sum_{\sigma\in B^{*}_n} y^{\fwex(\sigma)}t^{\nega(\sigma)}q^{\cro_B(\sigma)} .
\end{equation}
and the weight scheme for $\M_n^*$ is the following:
\begin{itemize}
\item $\rho(\U^{(h)})\in\{y^2,y^2q,\dots,y^2q^{h}\}\cup\{ytq^{h},ytq^{h+1},\dots,ytq^{2h}\}$,
\item $\rho(\L^{(h)})\in\{y^2,y^2q,\dots,y^2q^{h}\}\cup\{ytq^{h},ytq^{h+1},\dots,ytq^{2h}\}$ for $h\ge 1$\\
and $\rho(\L^{(0)})\in\{yt\}$ for $h=0$.
\item $\rho(\W^{(h)})\in\{1,q,\dots,q^{h-1}\}\cup\{ytq^{h},ytq^{h+1},\dots,ytq^{2h-1}\}$ for $h\ge 1$,
\item $\rho(\D^{(h+1)})\in\{1,q,\dots,q^{h}\}\cup\{ytq^{h+1},ytq^{h+2},\dots,ytq^{2h+1}\}$.
\end{itemize}

We shall establish an involution $\Psi_2:\M^{*}_n\rightarrow\M^{*}_n$ that changes the weight of a path by the factor $y^2q$, with the following set of restricted paths as the fixed points.

\begin{defi} \label{def:Gn} {\rm
Let $\G_n\subset\M^{*}_n$ be the subset consisting of the paths satisfying the following conditions. For $h\ge 0$,
\begin{itemize}
\item $\rho(\U^{(h)},\D^{(h+1)})=(y^2q^a,q^b)$ or $(ytq^{h+a},ytq^{h+1+b})$ for some $a,b\in\{0,1,\dots, h\}$, for any matching pair $(\U^{(h)},\D^{(h+1)})$,
	\item $\rho(\L^{(h)})\in\{ytq^{h}, ytq^{h+1},\dots,ytq^{2h}\}$,
	\item $\rho(\W^{(h)})\in\{ytq^{h},ytq^{h+1},\dots,ytq^{2h-1}\}$ for $h\ge 1$.
\end{itemize}
}
\end{defi}
	
Comparing the weight functions of the paths in $\G_n$ and in the set $\T^{*}_n$ given in Proposition \ref{pro:T*n-weighting},  the following result can be proved by the same argument as in the proof of Lemma \ref{lem:Fn=y^nRn}.
	
\begin{lem} \label{lem:Gn=y^nQn} We have
\[
\rho(\G_n)=y^n Q_n(t,q).
\]
\end{lem}	
	
We describe the construction of the map $\Psi_2:\M^{*}_n\rightarrow\M^{*}_n$.
	
\smallskip	
\noindent
{\bf Algorithm  B}

Given a path $\mu\in\M^{*}_n$, the corresponding path $\Psi_2(\mu)$ is constructed by the following procedure.
\begin{enumerate}
	\item[(B1)] If $\mu$ contains no straight level steps $\L^{(h)}$ with weight $y^2q^a$ or wavy level steps $\W^{(h)}$ with weight $q^{a-1}$ for any $a\in\{1,2,\dots, h\}$ then go to (B2). Otherwise, among such level steps find the first step $z$, say $z=\L^{(h)}$ ($\W^{(h)}$, respectively) with weight $\rho(z)=y^2q^a$ ($q^{a-1}$, respectively), then the path $\Psi_2(\mu)$ is obtained by replacing $z$ by $\W^{(h)}$ ($\L^{(h)}$, respectively) with weight $q^{a-1}$ ($y^2q^a$, respectively).
	\item[(B2)]  If $\mu$ contains no matching pairs $(\U^{(h)},\D^{(h+1)})$ with $\rho(\U^{(h)},\D^{(h+1)})=(y^2q^a, ytq^{h+1+b})$ or $(ytq^{h+a}, q^{b})$ for any $a,b\in\{0,1,\dots, h\}$ then go to (B3). Otherwise, among such pairs find the first pair with weight, say $(y^2q^a, ytq^{h+1+b})$ ($(ytq^{h+a}, q^{b})$, respectively) then the corresponding path is obtained by replacing the weight of the pair by $(ytq^{h+a}, q^{b})$ ($(y^2q^a, ytq^{h+1+b})$, respectively).
	\item[(B3)] Then we have $\mu\in\G_n$. Let $\Psi_2(\mu)=\mu$.
	\end{enumerate}

\smallskip
By the possible weighted steps of the paths in $\M^{*}_n$, we have the following result.
	
\begin{pro} \label{pro:map-psi2} The map $\Psi_2$ established by Algorithm B is an involution on the set $\M^{*}_n$ such that for any path $\mu\in\M^{*}_n$,  $\Psi_2(\mu)=\mu$ if $\mu\in\G_n$, and $\rho(\Psi_2(\mu))=y^2q\rho(\mu)$ or $y^{-2}q^{-1}\rho(\mu)$ otherwise.
\end{pro}	

Now we are ready to prove Theorem \ref{thm:signFwexB*}.

\smallskip
\noindent
\emph{Proof of Theorem \ref{thm:signFwexB*}.} For (i), the expression $\sum_{\sigma\in B^{*}_n} \big(-\frac{1}{q}\big)^{\lfloor \frac{\fwex(\sigma)}{2}\rfloor} t^{\nega(\sigma)}q^{\cro_B(\sigma)}$, by Eq.\,(\ref{eqn:expression-B*n}), equals the sum of the real part and the imaginary part multiplying $\sqrt{q}$ of the polynomial $\rho(\M^{*}_n)$ evaluated at $y=\sqrt{\frac{-1}{q}}$. By Proposition \ref{pro:map-psi2} and Lemma \ref{lem:Gn=y^nQn}, we have
\begin{align*}
\rho(\M^{*}_n)\big|_{y=\sqrt{\frac{-1}{q}}}
            &= \rho(\G_n)\big|_{y=\sqrt{\frac{-1}{q}}}  \\
            &= y^{n} Q_{n}(t,q)\big|_{y=\sqrt{\frac{-1}{q}}}\\
            &=\begin{cases}
			\left(\frac{-1}{q}\right)^{\frac{n}{2}}Q_n(t,q) & \mbox{, if $n$ is even};\\
			\sqrt{\frac{-1}{q}}\left(\frac{-1}{q}\right)^{\frac{n-1}{2}}Q_n(t,q) & \mbox{, if $n$ is odd.}
	  \end{cases}
\end{align*}
Hence 
\begin{equation} \label{eqn:evaluation-Bn*-floor}
\sum_{\sigma\in B^{*}_n}\left(-\frac{1}{q}\right)^{\lfloor \frac{\fwex(\sigma)}{2}\rfloor} t^{\nega(\sigma)}q^{\cro_B(\sigma)}
    =\left(-\frac{1}{q}\right)^{\lfloor\frac{n}{2}\rfloor}Q_n(t,q).  
\end{equation}
For (ii), the expression $\sum_{\sigma\in B^{*}_n} \big(-\frac{1}{q}\big)^{\lceil \frac{\fwex(\sigma)}{2}\rceil} t^{\nega(\sigma)}q^{\cro_B(\sigma)}$ can be obtained by subtracting the real part of $\rho(\M_n^{*})$ evaluated at $y=\sqrt{\frac{-1}{q}}$ with the imaginary part multplying $\sqrt{q}$. Therefore, we have  
\begin{equation} \label{eqn:evaluation-Bn*-ceil}
\sum_{\sigma\in B^{*}_n}\left(-\frac{1}{q}\right)^{\lceil \frac{\fwex(\sigma)}{2}\rceil} t^{\nega(\sigma)}q^{\cro_B(\sigma)}
    =\left(-\frac{1}{q}\right)^{\lceil\frac{n}{2}\rceil}  Q_n(t,q).  
\end{equation}	\qed

\smallskip
\noindent
\emph{Proof of Corollary \ref{thm: signFwexD*}}
Recall that $D^{*}_n\subset B^{*}_n$ is the subset consisting of the fixed point-free signed permutations with even number of negative entries.
Let $\M^{*'}_n\subset \M^{*}_n$ be the subset consisting of the paths whose weights contain even powers of $t$.
Notice that the bijection $\Gamma: B_n\rightarrow\M_n$ also induces a bijection between $D^{*}_n$ and $\M^{*'}_n$ such that
\begin{equation} \label{eqn:D*n-M*'n}
\sum_{\sigma\in D^{*}_n} y^{\fwex(\sigma)}t^{\nega(\sigma)}q^{\cro_B(\sigma)} = \rho(\M^{*'}_n).
\end{equation}

Consider the polynomial $\rho(\M^{*'}_n)$ evaluated at $y=\sqrt{\frac{-1}{q}}$.
By Proposition \ref{pro:map-psi2}, we observe that the involution $\Psi_2:\M^{*}_n\rightarrow\M^{*}_n$ preserves the powers of $t$ of the weight of the paths. Moreover, for the fixed points $\mu\in\G_n$ of the map $\Psi_2$, we observe that the power of $t$ of $\rho(\mu)$ has the same parity of $n$. By the expression in Eq.\,(\ref{eqn:evaluation-Bn*-floor}) and Eq.\,(\ref{eqn:evaluation-Bn*-ceil}), we have
\[
\sum_{\sigma\in D^{*}_n}\left(-\frac{1}{q}\right)^{\lfloor \frac{\fwex(\sigma)}{2}\rfloor} t^{\nega(\sigma)}q^{\cro_B(\sigma)}=\left\{
               \begin{array}{ll}
                \big(-\frac{1}{q}\big)^{\frac{n}{2}} Q_n(t,q) &\mbox{if $n$ is even,} \\
                0 &\mbox{if $n$ is odd.}
               \end{array}
               \right.
\] 	
Since $\lfloor\frac{\fwex(\sigma)}{2}\rfloor=\lceil\frac{\fwex(\sigma)}{2}\rceil=\frac{\fwex(\sigma)}{2}$ for $\sigma\in D_n^{*}$. The same result is obtained when we replace $\lfloor \frac{\fwex(\sigma)}{2}\rfloor$ with $\lceil \frac{\fwex(\sigma)}{2}\rceil$. Hence the proof of Corollary \ref{thm: signFwexD*} is completed. \qed

\chapter{Snakes and $(t,q)$-analogue of derivative polynomials}

In this chapter we shall give a combinatorial interpretation for the polynomials $Q_n(t,q)$ and $R_n(t,q)$ as the enumerators for variants of the snakes in signed permutations. We make use of a non-recursive approach, based on Flajolet's combinatorics of continued fractions, to establish a bijection between the snakes of $\SS^{0}_n$ ($\SS^{00}_{n+1}$, respectively) and the weighted bicolored Motzkin paths of $\T^{*}_n$ given in Proposition \ref{pro:T*n-weighting} ($\T_n$ given in Proposition \ref{pro:Tn-weighting}, respectively). The bijections are constructed in
the spirit of the classical Fran\c con--Viennot bijection \cite{FV} by encoding the sign changes and consecutive up-down patterns of snakes into weighted steps of the Motzkin paths.

\section{$\cs$-vectors and blocks}
Given a permutation $\pi=\pi_1\pi_2\cdots\pi_n\in\mathfrak{S}_n$, Arnol'd \cite{Arnold} devised the following rules to determine signs $\epsilon=\epsilon_1\epsilon_2\cdots\epsilon_n\in\{\pm\}^n$ such that $(\epsilon_1\pi_1,\epsilon_2\pi_2,\dots,\epsilon_n\pi_n)$ is a snake; see also \cite{JV}. For $2\le i\le n-1$,
\begin{itemize}
	\item[(R1)] if $\pi_{i-1}<\pi_i<\pi_{i+1}$ then $\epsilon_i\ne \epsilon_{i+1}$,
	\item[(R2)] if $\pi_{i-1}>\pi_i>\pi_{i+1}$ then $\epsilon_{i-1}\ne \epsilon_{i}$,
	\item[(R3)] if $\pi_{i-1}>\pi_i<\pi_{i+1}$ then $\epsilon_{i-1}= \epsilon_{i+1}$.
\end{itemize}

For a snake $\sigma=\sigma_1\sigma_2\dots\sigma_n\in \SS^{0}_n$ ($\SS^{00}_n$, respectively), we make the convention that $\sigma_0=0$ and $\sigma_{n+1}=(-1)^{n}(n+1)$ ($\sigma_0=\sigma_{n+1}=0$, respectively). Let $|\sigma|$ denote the permutation obtained from $\sigma$ by removing the negative signs of $\sigma$, i.e., $|\sigma|_i=|\sigma_i|$ for $0\le i\le n+1$.

Recall that the number of sign changes of $\sigma$ is given by
$\cs(\sigma):=\#\{i: \sigma_i\sigma_{i+1}<0, 0\le i\le n\}$.
For $1\le j\le n$, let $j=|\sigma|_i$ for some $i\in [n]$, and define $\cs(\sigma,j)$, the number of sign changes recorded by the element $j$, by
\[
\cs(\sigma,j)=\left\{
             \begin{array}{ll}
             0              &\mbox{if $|\sigma|_{i-1}>j<|\sigma|_{i+1}$ and $\sigma_{i-1}\sigma_i>0, \sigma_i\sigma_{i+1}>0$,} \\
             2              &\mbox{if $|\sigma|_{i-1}>j<|\sigma|_{i+1}$ and $\sigma_{i-1}\sigma_i<0, \sigma_i\sigma_{i+1}<0$,} \\
             1              &\mbox{if $|\sigma|_{i-1}<j<|\sigma|_{i+1}$ or $|\sigma|_{i-1}>j>|\sigma|_{i+1}$,} \\
             0              &\mbox{if $|\sigma|_{i-1}<j>|\sigma|_{i+1}$.}
             \end{array}
       \right.
\]
We call the sequence $(\cs(\sigma,1),\dots,\cs(\sigma,n))$ the $\cs$-\emph{vector} of $\sigma$.

Following the rules (R1)-(R3) and the condition $\sigma_1>0$, we observe that the signs of the entries $\sigma_1,\dots,\sigma_n$ of $\sigma$ can be recovered from left to right by $|\sigma|$ and the $\cs$-vector of $\sigma$.

\begin{lem} \label{lem:sign-changes} For any snake $\sigma\in\SS^{0}_n$ ($\SS^{00}_n$, respectively), the sign of each entry of $\sigma$ is uniquely determined by $|\sigma|$ and the vector $(\cs(\sigma,1), \dots, \cs(\sigma,n))$. Moreover, we have
\[
\cs(\sigma)=\sum_{j=1}^{n} \cs(\sigma,j).
\]
\end{lem}

\begin{proof} For the initial condition, we have $\sigma_0=0$ and $\sigma_1>0$. For $i\ge 2$, we determine the sign of $\sigma_i$ according to the following cases:

Case 1. $|\sigma|_{i-1}>|\sigma|_i$. There are two cases. If $|\sigma|_i>|\sigma|_{i+1}$ then by (R2) $\sigma_i$ has the opposite sign of $\sigma_{i-1}$.
Otherwise, $|\sigma|_i<|\sigma|_{i+1}$, and by (R3) $\sigma_i$ has the opposite (same, respectively) sign of $\sigma_{i-1}$ if $\cs(|\sigma|,|\sigma|_i)=2$ (0, respectively). Hence the sign change between $\sigma_{i-1}$ and $\sigma_i$ is recorded by $\cs(|\sigma|,|\sigma|_i)$.

Case 2. $|\sigma|_{i-1}<|\sigma|_i$. There are two cases. If $|\sigma|_{i-2}<|\sigma|_{i-1}$ then by (R1) $\sigma_i$ has the opposite sign of $\sigma_{i-1}$.
Otherwise, $|\sigma|_{i-2}>|\sigma|_{i-1}$, and by (R3) $\sigma_i$ has the opposite (same, respectively) sign of $\sigma_{i-1}$ if $\cs(|\sigma|,|\sigma|_{i-1})=2$ (0, respectively). Hence the sign change between $\sigma_{i-1}$ and $\sigma_i$ is recorded by $\cs(|\sigma|,|\sigma|_{i-1})$.

The assertions follow.
\end{proof}

\smallskip
\begin{exa} \label{exa:sign-snake} {\rm Given a snake $\sigma=((0),5,-2,4,-7,-1,-8,10,-9,6,3,(11))\in\SS^{0}_{10}$, note that $\cs(\sigma)=6$ and the $\cs$-vector of $\sigma$ is $(0,2,0,1,0,1,0,1,1,0)$.

On the other hand, given the permutation $|\sigma|=(5,2,4,7,1,8,10,9,6,3)\in\mathfrak{S}_{10}$ and the $\cs$-vector $(0,2,0,1,0,1,0,1,1,0)$, we observe that the snake $\sigma$ can be recovered, following the rules (R1)-(R3).
}
\end{exa}

Let $\mathfrak{S}^{0}_n$ and $\mathfrak{S}^{00}_n$ denote two `copies' of $\mathfrak{S}_n$ with the following convention
\begin{itemize}
\item $\pi_0=0$ and $\pi_{n+1}=n+1$ if $\pi\in\mathfrak{S}^{0}_n$,
\item $\pi_0=\pi_{n+1}=0$ if $\pi\in\mathfrak{S}^{00}_n$.
\end{itemize}

Given a permutation $\pi=\pi_1\pi_2\cdots\pi_n\in\mathfrak{S}^{0}_n$ or $\mathfrak{S}^{00}_n$, by a \emph{block} of $\pi$ restricted to $\{0,1,\dots,k\}$ we mean a maximal sequence of consecutive entries $\pi_i\pi_{i+1}\cdots\pi_j\subseteq\{0,1,\dots,k\}$ for some $i\le j$. For $0\le k\le n$, let $\alpha(\pi,k)$ be the number of blocks of $\pi$ restricted to $\{0,1\dots,k\}$, and let $\beta(\pi,k)$ be the number of such blocks on the right-hand side of the block containing the element $k$. By the convention on $\sigma_0$ and $\sigma_{n+1}$, for $k=0$ we have $\alpha(\pi,0)=1$ if $\pi\in\mathfrak{S}^{0}_n$, while $\alpha(\pi,0)=2$ if $\pi\in\mathfrak{S}^{00}_n$.

\begin{exa} \label{exa:alpha-beta-0} {\rm
Let $\pi=((0),5,2,4,7,1,8,10,9,6,3,(11))\in\mathfrak{S}^{0}_{10}$. Notice that $\alpha(\pi,6)=3$ and $\beta(\pi,6)=0$. The three blocks of $\pi$ restricted to $\{0,1,\dots,6\}$ are underlined as shown below.
\[
\begin{tabular}{cccccccccccc}
    (0) & 5 & 2 & 4 & 7 & 1 & 8 & 10 & 9 & 6 & 3 & (11) \\
    \cline{1-4} \cline{6-6}  \cline{10-11} \\

\end{tabular}
\]
For $0\le k\le 10$, the sequences of $\alpha(\pi,k)$ and $\beta(\pi,k)$ of $\pi$ are shown in Table \ref{tab:block-vector}.
}
\end{exa}

\begin{table}[ht]
\caption{The sequences $\alpha,\beta$ of $\pi=((0),5,2,4,7,1,8,10,9,6,3,(11))$.}
\begin{center}
\begin{tabular}{c|ccccccccccc}
    $k$           & 0 & 1 & 2 & 3 & 4 & 5 & 6 & 7 & 8 & 9 & 10 \\
    \hline
  $\alpha(\pi,k)$ & 1 & 2 & 3 & 4 & 4 & 3 & 3 & 2 & 2 & 2 & 1 \\
   $\beta(\pi,k)$ & 0 & 0 & 1 & 0 & 2 & 2 & 0 & 1 & 1 & 0 & 0
\end{tabular}
\end{center}
\label{tab:block-vector}
\end{table}

\section{The enumerator $Q_n(t,q)$ of $\SS^{0}_n$.}
In this section we shall encode the permutation $|\sigma|$ by a weighted bicolored Motzkin path. We construct a bijection between $\SS^{0}_n$ and the set of weighted bicolored Motzkin paths $\T_n^{*}$ whose generating fucntion of weight is equal $Q_n(t,q)$.   

We shall establish a map $\Lambda_1:\SS^{0}_n\rightarrow\T^{*}_n$ by the following procedure.

\smallskip
\noindent
{\bf Algorithm C.}

Given a snake $\sigma=\sigma_1\sigma_2\cdots\sigma_n\in\SS^{0}_n$, we associate $\sigma$ with a weighted path $\Lambda_1(\sigma)=z_1z_2\cdots z_n$.
For $1\le j\le n$, let $j=|\sigma|_i$ for some $i\in [n]$ and define the step $z_j$ according to the following cases:
\begin{enumerate}
\item if $|\sigma|_{i-1}>j<|\sigma|_{i+1}$ then $z_j=\U$ with weight
\[
\rho(z_j)=\left\{ \begin{array}{ll}
                   q^{\beta(|\sigma|,j)} &\mbox{ if $\sigma_{i-1}\sigma_i>0$ and $\sigma_i\sigma_{i+1}>0$,} \\
                   t^2q^{\beta(|\sigma|,j)+2\alpha(|\sigma|,j)-3} &\mbox{ if $\sigma_{i-1}\sigma_i<0$ and $\sigma_i\sigma_{i+1}<0$,}
                  \end{array}
          \right.
\]
\item if $|\sigma|_{i-1}<j<|\sigma|_{i+1}$ then $z_j=\L$ with weight $tq^{\beta(|\sigma|,j)+\alpha({|\sigma|,j})-1}$,
\item if $|\sigma|_{i-1}>j>|\sigma|_{i+1}$ then $z_j=\W$ with weight $tq^{\beta(|\sigma|,j)+\alpha({|\sigma|,j})-1}$,
\item if $|\sigma|_{i-1}<j>|\sigma|_{i+1}$ then $z_j=\D$ with weight $q^{\beta(|\sigma|,j)}$.
\end{enumerate}

\smallskip
Notice that the value $\cs(\sigma,j)$ is encoded as the power of $t$ in $\rho(z_j)$. For convenience, the element $j$ in (i) is called a \emph{valley}, in (ii) a \emph{double ascent}, in (iii) a \emph{double descent}, and in (iv) a \emph{peak} of $\sigma$.

\smallskip
\begin{exa} \label{exa:snake-0} {\rm
Take the snake $\sigma=((0),5,-2,4,-7,-1,-8,10,-9,6,3,(11))\in\SS^{0}_{10}$. The $\cs$-vector of $\sigma$ is given in Example \ref{exa:sign-snake} and the sequences  $\alpha$, $\beta$ of $|\sigma|$ are given in Example \ref{exa:alpha-beta-0}. We observe that $z_1=\U$ with weight $q^{\beta(|\sigma|,1)}=1$ since the element 1 is a valley without sign-changes, and that $z_2=\U$ with weight $t^2q^{\beta(|\sigma|,2)+2\alpha(|\sigma|,2)-3}=t^2q^4$ since the element 2 is a valley with sign-changes. The path $\Lambda_1(\sigma)$ is shown in Figure \ref{fig:snake-path}.
}
\end{exa}

\begin{figure}[ht]
\begin{center}
\psfrag{1}[][][0.95]{$1$}
\psfrag{t2q4}[][][0.95]{$t^2q^4$}
\psfrag{tq5}[][][0.95]{$tq^5$}
\psfrag{q2}[][][0.95]{$q^2$}
\psfrag{tq2}[][][0.95]{$tq^2$}
\psfrag{q}[][][0.95]{$q$}
\psfrag{tq}[][][0.95]{$tq$}
\includegraphics[width=3.0in]{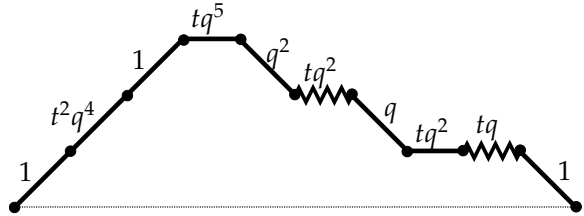}
\end{center}
\caption{\small The corresponding path of the snake in Example \ref{exa:snake-0}.}
\label{fig:snake-path}
\end{figure}

The constraints for the parameters $\alpha(|\sigma|,k)$ and $\beta(|\sigma|,k)$ of $|\sigma|$ are encoded in the height of the step $z_k\in\Lambda_1(\sigma)$.

\begin{lem} \label{lem:heights} For a snake $\sigma=\sigma_1\sigma_2\cdots\sigma_n\in\SS^{0}_n$, let $\Lambda_1(\sigma)=z_1z_2\cdots z_n$ be the path constructed by Algorithm C. For $1\le j\le n$, let $h_j$ be the height of the step $z_j$.  Then the following properties hold.
\begin{enumerate}
\item $h_j=\alpha(|\sigma|,j-1)-1$.
\item If $z_j=\W$ or $\D$ then $h_j\ge 1$ and $0\le \beta(|\sigma|,j)\le h_j-1$.
\item If $z_j=\U$ or $\L$ then $0\le \beta(|\sigma|,j)\le h_j$.
\end{enumerate}
\end{lem}

\begin{proof} For the initial condition, we have $\alpha(|\sigma|,0)=1$ and $h_1=0$. The first step is either $\U$ or $\L$ since the element 1 is either a valley or a double ascent. For $j\ge 1$, let $j=|\sigma|_i$ for some $i\in [n]$. By induction, we determine the height of $z_{j+1}$ according to the following cases:
\begin{itemize}
\item If $|\sigma|_{i-1}>j<|\sigma|_{i+1}$ then $z_j=\U$ and the element $j$ itself creates a block of $|\sigma|$ restricted to $\{0,1,\dots,j\}$. Hence $h_{j+1}=h_{j}+1=\alpha(|\sigma|,j-1)=\alpha(|\sigma|,j)-1$.
\item If $|\sigma|_{i-1}<j<|\sigma|_{i+1}$ or $|\sigma|_{i-1}>j>|\sigma|_{i+1}$ then $z_j=\L$  ($\W$, respectively) and the element $j$ is added to the block with $|\sigma|_{i-1}$ ($|\sigma|_{i+1}$, respectively). Hence $h_{j+1}=h_{j}=\alpha(|\sigma|,j-1)-1=\alpha(|\sigma|,j)-1$.
\item If $|\sigma|_{i-1}<j>|\sigma|_{i+1}$ then $z_j=\D$ and the element $j$ connects the adjacent two blocks. Hence $h_{j+1}=h_{j}-1=\alpha(|\sigma|,j-1)-2=\alpha(|\sigma|,j)-1$.
\end{itemize}
The assertion (i) follows.

(ii) If $z_j=\W$ or $\D$ then $j>|\sigma|_{i+1}$. The element $j$ is added to the block with $|\sigma|_{i+1}$, which is different from the block containing $(0)$. Then $\alpha(|\sigma|,j-1)\ge 2$ and hence $h_j\ge 1$. Moreover, there are at most $\alpha(|\sigma|,j-1)-2$ blocks on the right-hand side of the block containing $|\sigma|_{i+1}$. Hence $\beta(|\sigma|,j)\le h_j-1$.

(iii) If $z_j=\U$ or $\L$ then $j<|\sigma|_{i+1}$ and there are at most $\alpha(|\sigma|,j-1)-1$ blocks on the right-hand side of the block containing $j$. Hence $\beta(|\sigma|,j)\le h_j$.
\end{proof}

Comparing the weight function of the paths in $\T^{*}_n$ in Proposition \ref{pro:T*n-weighting} and the properties of $\Lambda_1(\sigma)$ in  Lemma \ref{lem:heights}, it follows that the path $\Lambda_1(\sigma)$ constructed by Algorithm C is a member of $\T^{*}_n$.

\smallskip
Next, we shall construct the map $\Lambda_1^{-1}:\T^{*}_n\rightarrow\SS^{0}_n$ by the following procedure.

\smallskip
\noindent
{\bf Algorithm D.}

Given a path $\mu=z_1z_2\cdots z_n\in\T^{*}_n$, we associate $\mu$ with a snake $\sigma'=\Lambda_1^{-1}(\mu)$.
For $1\le j\le n$, let $\cs(|\sigma'|,j)$ ($d_j$, respectively) be the power of $t$ ($q$, respectively) of $\rho(z_j)$, and let $h_j$ be the height of  $z_j$. To find $|\sigma'|$,
 we construct a sequence $\omega_0, \omega_1, \dots, \omega_n=|\sigma'|$ of words, where $\omega_j$ is the subword consisting of the blocks of $|\sigma'|$ restricted to $\{0,1,\dots,j\}$. The initial word $\omega_0$ is a singleton $(\sigma_0)$. For $j\ge 1$, the
word $\omega_j$ is constructed from $\omega_{j-1}$ and $\rho(z_j)$ according to the following cases:
\begin{enumerate}
\item $z_j=\U$. There are two cases. If $\cs(|\sigma'|,j)=0$,  let $\ell=d_j$. Otherwise $\cs(|\sigma'|,j)=2$ and let $\ell=d_j-2h_j-1$. Then the word $\omega_j$ is obtained from $\omega_{j-1}$ by inserting $j$ between the $\ell$th and the $(\ell+1)$st block from right as a new block.
\item $z_j=\L$ or $\W$. Then let $\ell=d_j-h_j$. The word $\omega_j$ is obtained from $\omega_{j-1}$ by appending $j$ to the right end (left end, respectively) of the $(\ell+1)$st block from right as a new member of the block if $z_j=\L$ ($\W$, respectively).
\item $z_j=\D$. Then let $\ell=d_j$. The word $\omega_j$ is obtained from $\omega_{j-1}$ by inserting $j$ between the $(\ell+1)$st and the $(\ell+2)$nd block from right and getting the two blocks combined.
\end{enumerate}
Then following the rules (R1)-(R3), we determine the signs of the elements of $\omega_n$ by the sequence $\cs(|\sigma'|,j)$ for $j=1,2,\dots, n$. Hence the requested snake $\sigma'$ is established.

\smallskip

In the following we give an interpretation of the sequences $\alpha$, $\beta$ in terms of three-term patterns of the permutation $|\sigma|$.

\begin{defi} {\rm
Let $\pi=\pi_1\pi_2\cdots\pi_n\in\mathfrak{S}^{0}_n$ or $\mathfrak{S}^{00}_n$. For $1\le i\le n$, we define
\begin{align*}
\acb(\pi,i) &= \#\{j: 0\le j<i-1 \mbox{ and }\pi_j<\pi_i<\pi_{j+1}\}, \\
\bca(\pi,i) &= \#\{j: i< j\le n \mbox{ and }\pi_j>\pi_i>\pi_{j+1}\}.
\end{align*}
Let also $\bca(\pi)=\sum_{i=1}^{n} \bca(\pi,i)$. For any snake $\sigma=\sigma_1\sigma_2\cdots\sigma_n\in\SS^{0}_n$ or $\SS^{00}_n$, we distinguish the following classes $X$, $Y$ and $Z$ of elements of $\sigma$:  (i) the valleys with sign changes,   (ii) the double ascents or double descents and (iii) the peaks, namely
\begin{align*}
X(\sigma) &=\{|\sigma|_i : |\sigma|_{i-1}>|\sigma|_i<|\sigma|_{i+1},  \sigma_{i-1}\sigma_i<0 \mbox{ and } \sigma_i\sigma_{i+1}<0\}, \\
Y(\sigma) &=\{|\sigma|_i : |\sigma|_{i-1}<|\sigma|_i<|\sigma|_{i+1} \mbox{ or } |\sigma|_{i-1}>|\sigma|_i>|\sigma|_{i+1}\}, \\
Z(\sigma) &=\{|\sigma|_i : |\sigma|_{i-1}<|\sigma|_i>|\sigma|_{i+1}\}.
\end{align*}
}
\end{defi}

The parameters $\alpha(|\sigma|,k)$ and $\beta(|\sigma|,k)$ of the snake $\sigma\in\SS^{0}_n$ or $\SS^{00}_{n}$ have the following properties.

\begin{lem} \label{lem:pattern} For $0\le k\le n$, we have
\begin{enumerate}
\item $\beta(|\sigma|,k)=\bca(|\sigma|,k)$,
\item $\alpha(|\sigma|,k)=\acb(|\sigma|,k)+\bca(|\sigma|,k)+1$.
\end{enumerate}
\end{lem}

\begin{proof} Suppose there are $\ell$ ($\ell'$, respectively) blocks on the right-hand (left-hand, respectively) side of the block containing $k$ when $|\sigma|$ is restricted to $\{0,1,\dots,k\}$. Then along with the element $k$, the two adjacent entries of $|\sigma|$ at the left (right, respectively) boundary of each block constitute a $\bca$-pattern ($\acb$-pattern, respectively). Hence $\bca(|\sigma|,k)=\ell$ and $\acb(|\sigma|,k)=\ell'$. The assertions follow.
\end{proof}

For example, let $\pi=((0),5,2,4,7,1,8,10,9,6,3,(11))\in\mathfrak{S}^{0}_{10}$. As shown in Example \ref{exa:alpha-beta-0}, $\alpha(\pi,6)=3$ and $\beta(\pi,6)=0$.
We have $\acb(\pi,6)=2$, where the two requested $\acb$-patterns are $(4,7,6)$ and $(1,8,6)$.

Following the weighting scheme given in Algorithm C,
we define the statistic $\pat_Q$ of a snake $\sigma\in\SS^{0}_n$ by
\begin{align} \label{eqn:patQ}
\begin{split}
\pat_Q(\sigma) &=\sum_{j\in X(\sigma)} 2\big(\acb(|\sigma|,j)+\bca(|\sigma|,j)\big)-\#X(\sigma) \\
               &\qquad\qquad +\sum_{j\in Y(\sigma)} \big( \acb(|\sigma|,j)+\bca(|\sigma|,j)\big).
\end{split}
\end{align}
By Lemmas \ref{lem:heights} and \ref{lem:pattern} and Proposition \ref{pro:T*n-weighting},
we have the following result.

\smallskip
\begin{thm} \label{thm:T*->snake-0} The map $\Lambda_1$ established by Algorithm C is a bijection between $\SS^{0}_n$ and $\T^{*}_n$ such that
\[
\sum_{\sigma\in \SS_n^{0}} t^{\cs(\sigma)}q^{\bca(|\sigma|)+\pat_Q(\sigma)}=Q_n(t,q).
\]
\end{thm}

\section{The enumerator $R_n(t,q)$ of $\SS^{00}_{n+1}$.} 
In this section we shall apply a similar procedure to $R_n(t,q)$ and $\SS^{00}_{n+1}$. We establish a map $\Lambda_2:\SS^{00}_{n+1}\rightarrow\T_n$ by the same method as in Algorithm C with a modification on the weighting scheme.

Given a snake $\sigma=\sigma_1\sigma_2\cdots\sigma_{n+1}\in\SS^{00}_{n+1}$, recall that the $\cs$-vector of $\sigma$ and the parameters $\alpha(|\sigma|,k)$, $\beta(|\sigma|,k)$ for $k=1,2,\dots, n$ are computed under the convention $\sigma_0=\sigma_{n+2}=0$.

\begin{exa} \label{exa:alpha-beta-00} {\rm
Let $\sigma=((0),5,-2,4,-7,-1,-8,11,-9,6,3,10,(0))\in\SS^{00}_{11}$. Notice that $\alpha(|\sigma|,6)=4$ and $\beta(|\sigma|,6)=1$ as shown below.
\[
\begin{tabular}{ccccccccccccc}
    (0) & 5 & 2 & 4 & 7 & 1 & 8 & 11 & 9 & 6 & 3 & 10 & (0) \\
    \cline{1-4} \cline{6-6}  \cline{10-11} \cline{13-13} \\
\end{tabular}
\]
For $0\le k\le 10$, the sequences $\alpha(|\sigma|,k)$ and $\beta(|\sigma|,k)$ of $|\sigma|$ are shown in Table \ref{tab:block-vector-00}.
}
\end{exa}

\begin{table}[ht]
\caption{The $\alpha$ and $\beta$ vectors of $|\sigma|$.}
\begin{center}
\begin{tabular}{c|cccccccccccc}
    $k$           & 0 & 1 & 2 & 3 & 4 & 5 & 6 & 7 & 8 & 9 & 10 & 11 \\
    \hline
  $\alpha(|\sigma|,k)$ & 2 & 3 & 4 & 5 & 5 & 4 & 4 & 3 & 3 & 3 & 2  &    \\
   $\beta(|\sigma|,k)$ &   & 1 & 2 & 1 & 3 & 3 & 1 & 2 & 2 & 1 & 0  &
\end{tabular}
\end{center}
\label{tab:block-vector-00}
\end{table}

We associate the snake $\sigma$ with a weighted path $\Lambda_2(\sigma)=z_1z_2\cdots z_n$ by the following procedure.

\smallskip
\noindent
{\bf Algorithm C'.}

For $1\le j\le n$, let $j=|\sigma|_i$ for some $i\in [n+1]$ and define the step $z_j$ according to the following cases:
\begin{enumerate}
\item if $|\sigma|_{i-1}>j<|\sigma|_{i+1}$ then $z_j=\U$ with weight
\[
\rho(z_j)=\left\{ \begin{array}{ll}
                   q^{\beta(|\sigma|,j)-1} &\mbox{ if $\sigma_{i-1}\sigma_i>0$ and $\sigma_i\sigma_{i+1}>0$,} \\
                   t^2q^{\beta(|\sigma|,j)+2\alpha(|\sigma|,j)-5} &\mbox{ if $\sigma_{i-1}\sigma_i<0$ and $\sigma_i\sigma_{i+1}<0$,}
                  \end{array}
          \right.
\]
\item if $|\sigma|_{i-1}<j<|\sigma|_{i+1}$ then $z_j=\L$ with weight $tq^{\beta(|\sigma|,j)+\alpha({|\sigma|,j})-2}$,
\item if $|\sigma|_{i-1}>j>|\sigma|_{i+1}$ then $z_j=\W$ with weight $tq^{\beta(|\sigma|,j)+\alpha({|\sigma|,j})-2}$,
\item if $|\sigma|_{i-1}<j>|\sigma|_{i+1}$ then $z_j=\D$ with weight $q^{\beta(|\sigma|,j)}$.
\end{enumerate}

\smallskip
For example, take the snake $\sigma=((0),5,-2,4,-7,-1,-8,11,-9,6,3,10,(0))\in\SS^{00}_{11}$. From the parameters $\alpha(|\sigma|,k)$, $\beta(|\sigma|,k)$ of $\sigma$ given in Example \ref{exa:alpha-beta-00}, the corresponding path $\Lambda_2(\sigma)$ is shown in Figure \ref{fig:snake-path-00}.

\begin{figure}[ht]
\begin{center}
\psfrag{1}[][][0.95]{$1$}
\psfrag{t2q5}[][][0.95]{$t^2q^5$}
\psfrag{tq6}[][][0.95]{$tq^6$}
\psfrag{q3}[][][0.95]{$q^3$}
\psfrag{tq3}[][][0.95]{$tq^3$}
\psfrag{q2}[][][0.95]{$q^2$}
\psfrag{tq2}[][][0.95]{$tq^2$}
\includegraphics[width=3.0in]{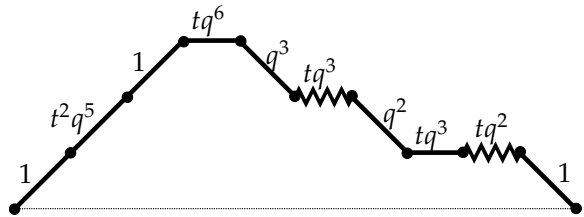}
\end{center}
\caption{\small The corresponding path of the snake $\sigma\in\SS^{00}_{11}$.}
\label{fig:snake-path-00}
\end{figure}

\begin{lem} \label{lem:heights-00} For a snake $\sigma=\sigma_1\sigma_2\cdots\sigma_{n+1}\in\SS^{00}_{n+1}$, let $\Lambda_2(\sigma)=z_1z_2\cdots z_n$ be the path constructed by Algorithm C'. For $1\le j\le n$, let $h_j$ be the height of the step $z_j$. Then the following properties hold.
\begin{enumerate}
\item $h_j=\alpha(|\sigma|,j-1)-2$.
\item If $z_j=\D$ then $h_j\ge 1$.
\item If $z_j=\W$ or $\D$ then $0\le \beta(|\sigma|,j)\le h_j$.
\item If $z_j=\U$ or $\L$ then $1\le \beta(|\sigma|,j)\le h_j+1$.
\end{enumerate}
\end{lem}

\begin{proof}
By the initial condition $\alpha(|\sigma|,0)=2$, we have $h_1=0$. The first step is $\U$, $\L$ or $\W$. Note that $z_1=\L$ ($\W$, respectively) if $|\sigma|_1=1$ ($|\sigma|_{n+1}=1$, respectively) and $z_1=\U$ if $|\sigma|_i=1$ for some $i\in\{2,\dots,n\}$. For $j\ge 1$, let $j=|\sigma|_i$ for some $i\in [n+1]$.
The assertion (i) can be proved by the same argument as in the proof of (i) of Lemma \ref{lem:heights}.

(ii) Note that $\alpha(|\sigma|,j)\ge 2$ since the greatest element $n+1$ is always absent. If $z_j=\D$ then $\alpha(|\sigma|,j-1)\ge 3$ and hence $h_j\ge 1$.

(iii) If $z_j=\W$ or $\D$ then $j>|\sigma|_{i+1}$ and the element $j$ is added to the block with $|\sigma|_{i+1}$. Then there are at most $\alpha(|\sigma|,j-1)-2$ blocks on the right-hand side of the block containing $|\sigma|_{i+1}$. Hence $0\le\beta(|\sigma|,j)\le h_j$.

(iv) If $z_j=\U$ or $\L$ then $j<|\sigma|_{i+1}$ and there are at least one block and at most $\alpha(|\sigma|,j-1)-1$ blocks on the right-hand side of the element $j$. Hence $1\le\beta(|\sigma|,j)\le h_j+1$.
\end{proof}

\smallskip
Comparing the weight function of the paths in $\T_n$ in Proposition \ref{pro:Tn-weighting} and the properties of $\Lambda_2(\sigma)$ in  Lemma \ref{lem:heights-00}, it follows that the path $\Lambda_2(\sigma)$ constructed by Algorithm C' is a member of $\T_n$.

\smallskip
To find $\Lambda^{-1}_2$, given a path $\mu=z_1z_2\cdots z_n\in\T_n$, we shall construct the corresponding snake $\sigma'=\Lambda_2^{-1}(\mu)$ by the following procedure.

\smallskip
\noindent
{\bf Algorithm D'.}

For $1\le j\le n$, let $\cs(|\sigma'|,j)$ ($d_j$, respectively) be the power of $t$ ($q$, respectively) of $\rho(z_j)$ and let $h_j$ be the height of $z_j$. To find $|\sigma'|$,
 we construct a sequence $\omega_0,\omega_1,\dots,\omega_n$ of words, where $\omega_j$ is the subword consisting of the blocks of $|\sigma'|$ restricted to $\{0,1,\dots,j\}$. The last word $\omega_n$ contains exactly two blocks, and the requested permutation $|\sigma'|$ is obtained from $\omega_n$ by inserting the element $n+1$ between the two blocks.

The initial word $\omega_0$ consists of the two blocks $(\sigma_0)$ and $(\sigma_{n+1})$. For $j\ge 1$, the
word $\omega_j$ is constructed from $\omega_{j-1}$ and $\rho(z_j)$ according to the following cases:
\begin{enumerate}
\item $z_j=\U$. There are two cases. If $\cs(|\sigma'|,j)=0$,  let $\ell=d_j+1$. Otherwise $\cs(|\sigma'|,j)=2$ and let $\ell=d_j-2h_j-1$. Then the word $\omega_j$ is obtained from $\omega_{j-1}$ by inserting $j$  between the $\ell$th and the $(\ell+1)$st block from right as a new block.
\item $z_j=\L$ or $\W$. Then let $\ell=d_j-h_j$. The word $\omega_j$ is obtained from $\omega_{j-1}$ by appending $j$ to the right end (left end, respectively) of the $(\ell+1)$st block from right as a new member of the block if $z_j=\L$ ($\W$, respectively).

\item $z_j=\D$. Then let $\ell=d_j$. The word $\omega_j$ is obtained from $\omega_{j-1}$ by inserting $j$ to $\omega_{j-1}$ between the $(\ell+1)$st and the $(\ell+2)$nd block from right and getting the two blocks combined.
\end{enumerate}
Then the signs of the elements of the requested snake $\sigma'$ can be determined by $|\sigma'|$ and the sequence $\cs(|\sigma'|,j)$ for $j=1,2,\dots,n$.

\smallskip
Following the weighting scheme given in Algorithm C', we define the statistic $\pat_R$ of a snake $\sigma\in\SS^{00}_{n+1}$ by
\begin{align} \label{eqn:patR}
\begin{split}
\pat_R(\sigma) &=\sum_{j\in X(\sigma)} 2\big(\acb(|\sigma|,j)+\bca(|\sigma|,j)-1\big) \\
               &\qquad\qquad +\sum_{j\in Y(\sigma)} \big(\acb(|\sigma|,j)+\bca(|\sigma|,j)\big)+\#Z(\sigma).
\end{split}
\end{align}
Notice that $\sum_{j=1}^{n} (\beta(|\sigma|,j)-1)=\bca(|\sigma|)-n$ and that $Z(\sigma)$ contains the element $n+1$, which is not involved in step (iv) of Algorithm C'.
By Lemmas \ref{lem:heights-00} and \ref{lem:pattern} and Proposition \ref{pro:Tn-weighting},
we have the following result.

\smallskip
\begin{thm} \label{thm:T->snakes-00} The map $\Lambda_2$ established by Algorithm C' is a bijection between $\SS^{00}_{n+1}$ and $\T_n$ such that
\[
\sum_{\sigma\in \SS^{00}_{n+1}} t^{\cs(\sigma)}q^{\bca(|\sigma|)+\pat_R(\sigma)-n-1}=R_n(t,q).
\]
\end{thm}

\chapter{Discussions}
In this chapter we discuss some possible direction for future research.\\
\begin{enumerate}
\item In this work, we give various signed countings on type B and D permtuations and derangements. When consindering the $(t,q)$-analogs, we obtain the $(t,q)$-derivative polynomials $Q_n(t,q)$ and $R_n(t,q)$. In $Q_n(t,q)$ the power of $t$ counts the sign changing  and that of $q$ counts $\bca(|\cdot|)+\pat_Q$. In $R_n(t,q)$ the power of $t$ counts the sign changing  and that of $q$ counts $\bca(|\cdot|)+\pat_R-n-1$. However, there are additional signed counting identities in corollary \ref{cor:sign-S_n^D} in which type D Springer number $S_n^D$ appears. It is interesting to see whether after taking parameters $t$ and $q$ into consideration the signed counting identities are some enumerators of snakes of type D. In fact, we already have some observations.\\
Consider $\sum_{\sigma\in B_n-B_n^{*}}(-1)^{\lfloor\frac{\fwex(\sigma)}{2}\rfloor}t^{\nega(\sigma)}$ in the case of even $n$ and $\sum_{\sigma\in B_n-B_n^{*}}(-1)^{\lceil\frac{\fwex(\sigma)}{2}\rceil}t^{\nega(\sigma)}$ in the case of odd $n$. With the aid of Python, we have
\[
	\begin{array}{l|l}
	n=2 & -\underline{2t+1}\\
	n=3 & \underline{6t^2-3t}+2\\
	n=4 & \underline{24t^3-12t^2}+\underline{16t - 5}\\
	n=5 & -\underline{120t^4 + 60t^3}-\underline{120t^2 + 45t} - 16\\
	n=6 & -\underline{720t^5 + 360t^4} -\underline{960t^3 + 390t^2} -\underline{272t + 61}\\
	n=7 & \underline{5040t^6 - 2520t^5} + \underline{8400t^4 - 3570t^3} + \underline{3696t^2 - 1113t} + 272.
	\end{array}
\]
On the other hand, the set of snakes of type D is defined to be
\[
	\SS_n^D=\{\sigma\in D_n|~\sigma_1+\sigma_2<0 \mbox{ and }\sigma_1>\sigma_2<\sigma_3>\ldots\}.
\]
If we set $\sigma_0>0$ and define $\cs_D=\#\{i\in [n-1]\cup\{0\}| \sigma_i\cdot\sigma_{i+1}<0\}$. The computer shows the polynomials from $n=2$ to $n=7$ are
\[
	\begin{array}{l|l}
	n=2 & t\\
	n=3 & 3t^2+\underline{2t}\\
	n=4 & 12t^3+\underline{6t^2+5t}\\
	n=5 & \ 60t^4+\underline{30t^3+45t^2}+\underline{16t}\\
	n=6 & 360t^5+\underline{180t^4+390t^3}+\underline{150t^2+61t}\\
	n=7 & 2520t^6+\underline{1260t^5+3570t^4}+\underline{1470t^3+1113t^2}+\underline{272t}.
	\end{array}	
\]
Observe that in the two tables, if we add the coefficients in each underline terms together, the polynomials in the two table have the same distribution. With this observation, it is reasonable to expect there are some relations between the distribution of signed changing on $\SS_n^D$ and the signed counting $\sum_{B_n-B_n^*}(-1)^{\lfloor\frac{\fwex(\sigma)}{2}\rfloor}t^{\nega(\sigma)}$ hiding behind the phenomenon. We would like to know what makes the phenomenon occurs.


\item Recall that in the case of type A there is a notion in some sense dual to crossings which is called nestings. The joint distribution of crossing number and nesting number are symmetric in $\mathfrak{S}_n$, i.e. $(\cro, \nest)$ has the same distribution as $(\nest,\cro)$ in $\mathfrak{S}_n$ (see \cite{Corteel-cross}) for details).    
A type B analogous of this result had been proved in 2011 by Hamdi \cite{Hamdi}. A type B nesting is defined as the following
\begin{defi}[Nestings of type B]
For $\sigma=\sigma_1\sigma_2\cdots\sigma_n\in B_n$, a \emph{nesting} of $\sigma$ is a pair $(i,j)$  with $i,j\ge 1$ such that
\begin{itemize}
  \item $i<j\le \sigma_j< \sigma_i$ or
  \item $-i<j\le \sigma_j< -\sigma_i$ or
  \item $j>i>\sigma_i>\sigma_j$.
\end{itemize}
Denote $\nestB(\sigma)$ the number of nestings in $\sigma$.
\end{defi}

We replace the $q$-derivative in (\ref{q-derivative}) as $(p,q)$-derivative $D_{p,q}$
\begin{equation*}
(D_{p,q}f)(t) := \frac{f(pt)-f(qt)}{(p-q)t},
\end{equation*}
then $D_{p,q}(t^n)=[n]_{p,q}t^{n-1}$ where $[n]_{p,q}=p^{n-1}+p^{n-2}q+\ldots+pq^{n-2}+q^{n-1}$. Similarly, we can also define the $(p,q)$-derivative polynomials $Q_n(t,p,q)$ and $R_n(t,p,q)$. 
\begin{conj}
For $n\ge 1$, we have
\begin{enumerate}
\item  ${\displaystyle
\sum_{\sigma\in B_n}(-1)^{\lfloor \frac{\fwex(\sigma)}{2}\rfloor} t^{\nega(\sigma)}p^{\nest_B(\sigma)}q^{\cro_B(\sigma)}
=\begin{cases}
	(-1)^{\frac{n}{2}}(t+1) R_{n-1}(t,p,q) &\mbox{, if $n$ is odd;}\\
	(-1)^{\frac{n-1}{2}}(t-1) R_{n-1}(t,p,q) &\mbox{, if $n$ is even.}
\end{cases}
}$
\item ${\displaystyle
\sum_{\sigma\in B_n}(-1)^{\lceil \frac{\fwex(\sigma)}{2}\rceil} t^{\nega(\sigma)}p^{\nest_B(\sigma)}q^{\cro_B(\sigma)}=\begin{cases}
	(-1)^{\frac{n}{2}}(t-1)R_{n-1}(t,p,q) & \mbox{ if $n$ is even;}\\
	(-1)^{\frac{n+1}{2}}(t+1)R_{n-1}(t,p,q) & \mbox{ if $n$ is odd.} 
			\end{cases}	.
}$
\end{enumerate}
\end{conj}
If the conjecture holds, naturally we have the derivation of type D from the conjecture. However, we haven't formulate the conjecture of similar signed counting identities for set $B_n^*$ of type B derangements.

\item Another possible direction for future research is to generalize our signed counting results to colored permutations $\mathbb{Z}_r\wr\mathfrak{S}_n$. The results without parameter $t$,$q$ has been prove by Athanasiadis \cite{A} as a byproduct of studying the $\gamma$-nonnegativity on Eulerian polynomial of $\mathbb{Z}_r\wr\mathfrak{S}_n$. Other paper which might be useful is \cite{Shin2014}.
\end{enumerate}



\begin{thebibliography}{99}

\bibitem{Arnold} V.I. Arnol'd, The calculus of snakes and the combinatorics of Bernoulli, Euler, and Springer numbers for Coxeter groups, Russian Math. Surveys 47 (1992) 1--51.

\bibitem{A} C.A. Athanasiadis, Edgewise subdivisions, local $h$-polynomials and excedances in the wreath product $\mathbb{Z}_r \wr \mathfrak{S}_n$,  SIAM J. Discrete Math. 28 (2014) 1479--1492.

\bibitem{Chebikin} D. Chebikin, Variations on descents and inversions in permutations, Electron. J. Combin. 15 (2008) \#R132.

\bibitem{Corteel-cross} S. Corteel, Crossings and alignments of permutations, Adv. Appl. Math. 38(2) (2007) 149--163.

\bibitem{CJK} S. Corteel, M. Josuat-Verg\`es, J.S. Kim, Crossings of signed permutations and $q$-Eulerian
numbers of type B, J. Combin. 4(2) (2013) 191--228.

\bibitem{CJW} S. Corteel, M. Josuat-Verg\`es, L.K. Williams, Matrix Ansatz, orthogonal polynomials and
permutations, Adv. Appl. Math. 46 (2011) 209--225.

\bibitem{CP2018} S. Cho, K. Park, A Combinatorial Proof of a Symmetry of $(t,q)$-Eulerian Numbers of Type B and Type D, Ann. Comb. 22 (2018) 99--134.

\bibitem{EFHL2018} S.-P. Eu, T.-S. Fu, H. -C. Hsu, H.-C. Liao, Signed countings of types B and D permutations and $t,q$-Euler numbers, Adv. Appl. Math. 97 (2018) 1-26.

\bibitem{Euler} L. Euler, Institutiones calculi differentialis cum eius usu in analysi finitorum ac doctrina serierum, in  \textit{Academiae Imperialis Scientiarum Petropolitanae}, St. Petersburg 1755 (Part II, chapter 7:  Methodus summandi superior ulterius promota.) 

\bibitem{Flaj}  P. Flajolet, Combinatorial aspects of continued fractions, Discrete Math. 32 (1980) 125--161.

\bibitem{FH} D. Foata, G.-N. Han, The $q$-tangent and $q$-secant numbers via basic eulerian polynomials, Proc. Amer. Math. Soc. 138(2)(2010) 385--393.

\bibitem{Foata-Schu} D. Foata, M.-P. Sch\"utzenberger,
Th\'eorie g\'eom\'etrique des polyn\^omes eul\'eriens, in: Lecture Notes in Mathematics vol. 138, Springer-Verlag, Berlin, 1970.

\bibitem{FV} J. Fran\c con, X. Viennot, Permutations selon leurs pics, creux, doubles mont\'ees et double descentes, nombres d'Euler et nombres de Genocchi, Discrete Math. 28 (1979) 21--35.

\bibitem{Foata-Zeil} D. Foata, D. Zeilberger, Denert's permutation statistic is indeed Euler-Mahonian, Stud. Appl. Math. 83(1) (1990) 31--59.

\bibitem{Hamdi} A. Hamdi, Symmetric distribution of crossings and nestings in permutations of type B, Electron. J. Combin. 18 (2011), P200.

\bibitem{Han-Zeng} G.-N. Han, A. Randrianarivony, J. Zeng, Un autre q-analogue des nombres d'Euler, S\'{e}m. Lothar. Combin, B42e (1999).

\bibitem{Hoffman} M.E. Hoffman, Derivative polynomials, Euler polynomials, and associated integer sequences,
Electron. J. Combin. 6 (1999), R21.

\bibitem{JVerges2010} M. Josuat-Verg\`{e}s, A $q$-enumeration of alternating permutations, European J. Combin. 31 (2010) 1892--1906.

\bibitem{JV} M. Josuat-Verg\`es, Enumeration of snakes and cycle-alternating permutations, Australas. J. Combin. 60(3) (2014) 279--305.

\bibitem{Roselle1968} P.D. Roselle, Permutations by number of rises and successions,  Proc. Amer. Math. Soc. 19 (1968) 8--16.

\bibitem{Shin2010} H. Shin, J. Zeng, The $q$-tangent and $q$-secant numbers via continued fractions, European J. Combin. 31 (2010) 1689--1705.

\bibitem{Shin2014} H. Shin, J. Zeng, Symmetric unimodal expansions of excedances in colored permutations, European J. Combin. 52, part A (2016) 174--196.


\bibitem{Springer} T.A. Springer, Remarks on a combinatorial problem, Nieuw Arch. Wisk. 19(3) (1971), 30–36.

\bibitem{Stanley1} R. P. Stanley, Enumerative Combinatorics. Vol. 1, vol. 49 of \textit{Cambridge Studies in Adcanced Mathematics.} Cambridge University Press, Cambridge, 1997.

\bibitem{WilliamsGrs} L.K. Williams, Enumeration of totally positive Grassmann cells, Adv. Math. 190 (2005) 319--342.

\end{thebibliography}
\end{document}